\documentclass[11pt,letterpaper]{amsart}
\pdfoutput=1

\usepackage{amsopn,amssymb}

\usepackage{eucal}

\usepackage{nicefrac}

\usepackage[all]{xy}

\usepackage{lpic}

\usepackage{float}

\usepackage{caption}
\newcommand{\into}{\hookrightarrow}
\newcommand{\Z}{\mathbb{Z}}
\newcommand{\R}{\mathbb{R}}
\newcommand{\C}{\mathbb{C}}
\DeclareMathOperator{\im}{Im}
\DeclareMathOperator{\re}{Re}

\renewcommand{\Re}{\re}
\renewcommand{\Im}{\im}

\renewcommand{\O}{\mathcal{O}}

\newcommand{\CP}{\mathbb{CP}}
\newcommand{\RP}{\mathbb{RP}}
\newcommand{\QP}{\mathbb{QP}}
\renewcommand{\H}{\mathbb{H}}
\newcommand{\ML}{\mathcal{ML}}
\newcommand{\PML}{\mathbb{P}\mathcal{ML}}
\newcommand{\X}{\mathcal{X}}

\newcommand{\GF}{\mathcal{GF}}
\newcommand{\QF}{\mathcal{QF}}
\newcommand{\T}{\mathcal{T}}
\newcommand{\Rh}{\mathcal{R}}

\DeclareMathOperator{\PSL}{\mathrm{PSL}}
\DeclareMathOperator{\PGL}{\mathrm{PGL}}
\DeclareMathOperator{\SL}{\mathrm{SL}}
\DeclareMathOperator{\Hom}{Hom}

\newcommand{\pmat}[4]{\begin{pmatrix}{#1} & {#2}\\{#3} & {#4}\end{pmatrix}}

\numberwithin{equation}{section}

\theoremstyle{plain}
\newtheorem{thm}{Theorem}[section]
\newtheorem{cor}[thm]{Corollary}

\newtheorem{lemma}[thm]{Lemma}
\newtheorem{prop}[thm]{Proposition}
\newenvironment{definition}[1][Definition.]{\begin{trivlist}
\item[\hskip \labelsep {\bfseries #1}]}{\end{trivlist}}

\theoremstyle{definition}

\theoremstyle{definition}

\begin{document}
\title[a family of non-injective skinning maps with critical points]{a family of non-injective skinning maps\\ with critical points}
\author{Jonah Gaster}

\address{
Department of Mathematics, Statistics, and Computer Science\\
University of Illinois - Chicago\\
\tt{gaster@math.uic.edu}
}

\date{December 21, 2012}

\begin{abstract} 
{Certain classes of 3-manifolds, following Thurston, give rise to a `skinning map', a self-map of the Teichm\"{u}ller space of the boundary. This paper examines the skinning map of a 3-manifold $M$, a genus-2 handlebody with two rank-1 cusps. We exploit an orientation-reversing isometry of $M$ to conclude that the skinning map associated to $M$ sends a specified path to itself, and use estimates on extremal length functions to show non-monotonicity and the existence of a critical point. A family of finite covers of $M$ produces examples of non-immersion skinning maps on the Teichm\"{u}ller spaces of surfaces in each even genus, and with either $4$ or $6$ punctures.} 
\end{abstract}
\maketitle
\section{Introduction}
Thurston introduced the {\bf skinning map} as a tool for locating hyperbolic structures on certain closed 3-manifolds, an integral part of his celebrated proof of the Hyperbolization Theorem for Haken manifolds ~\cite{thurston}. Apart from a class of simple cases, explicit examples of skinning maps remain unexplored. One reason for this is that the skinning map uses the deformation theory of Ahlfors and Bers to pass back and forth from a conformal structure on the boundary of a 3-manifold to a hyperbolic structure on its interior. Like the uniformization theorem of Riemann surfaces, an explicit formula for the resulting map is typically out of reach.

Let $M$ be an orientable compact 3-manifold with nonempty incompressible boundary $\Sigma:=\partial M$, such that the interior of $M$ admits a geometrically finite hyperbolic structure. Given a point $X\in\T(\Sigma)$ in the Teichm\"{u}ller space of $\Sigma$, the 3-manifold $M$ has a quasi-Fuchsian cover with $X$ on one side and $\overline{Y}\in\T(\overline{\Sigma})$ on the other, where $\overline{\cdot}$ indicates a reversal of orientation. The skinning map of $M$, $\sigma_{M}$, is given by $\sigma_{M}(X)=\overline{Y}$. Thurston's key result about skinning maps is the Bounded Image Theorem: if $M$ is acylindrical, then $\sigma_{M}(\T(\Sigma))$ is contained in a compact set. 

Thurston described how to locate hyperbolic structures on a class of closed Haken 3-manifolds by iteration of $\tau\circ\sigma_{M}$, where $\tau$ is the `gluing map', an isometry from $\T(\overline{\Sigma})$ to $\T(\Sigma)$. There are further questions, and results, about `how effectively' iteration of $\tau\circ\sigma_{M}$ locates fixed points. ~\cite{kent} shows that the diameter of the image of $\sigma_{M}$ is controlled by constants depending only on the volume of $M$ and the topology of $\Sigma$. If $M$ is without rank-1 cusps ~\cite{dumas-kent} shows that $\sigma_{M}$ is non-constant. This result has an improvement in ~\cite{dumas}, which shows that $\sigma_{M}$ is open and finite-to-one. 

The only example for which there is an explicit formula for $\sigma_{M}$ is when $M$ is an interval bundle over a surface, in which case $\sigma_{M}$ is a diffeomorphism. Thus it is consistent with the current literature to ask: 

\begin{itemize}
\item Are skinning maps necessarily diffeomorphisms onto their images? 
\item Are they necessarily immersions? 
\end{itemize}

In this paper, we present a negative answer to these questions. The reader should note that we work in the category of pared 3-manifolds and hyperbolic structures with rank-1 cusps. This allows low-dimensional calculations, in which $\sigma_{M}$ is `simply' a holomorphic map $\sigma_{M}:\H\to\overline{\H}$. 

\begin{thm}{There exists a pared 3-manifold $M=(H_{2},P)$, where $H_{2}$ is a genus-2 handlebody and $P$ consists of two rank-1 cusps, whose skinning map $\sigma_{M}$ is not one-to-one and has a critical point.}
\label{thm}
\end{thm}

Briefly, the proof of Theorem \ref{thm} uses an orientation-reversing isometry of $M$ to conclude that the skinning map sends a certain real 1-dimensional submanifold of $\T(\Sigma)$ to itself. A parametrization of this line by extremal length shows that this restricted map is not one-to-one, and hence the skinning map has a critical point. That is, this proof produces a (real) 1-parameter path of quasi-Fuchsian groups $\left\{Q\left(X_{t},\sigma_{M}\left(X_{t}\right)\right)\right\}$, where $X_{t}$ and $\sigma_{M}(X_{t})$ are both contained in a line in $\T(\Sigma)$, such that $t\mapsto X_{t}$ is injective and $t\mapsto\sigma_{M}\left(X_{t}\right)$ is not.

By passing to finite covers of the manifold from Theorem \ref{thm}, this example implies a family of examples of skinning maps with critical points, on Teichm\"{u}ller spaces of arbitrarily high dimension. Below, $\Sigma_{g,n}$ indicates a topological surface of genus $g$ with $n$ punctures.

\begin{cor}{There exists a family of pared 3-manifolds $\left\{M_{n}=\left(H_{n},P_{n}\right)\right\}^{\infty}_{n=2}$, each admitting a geometrically finite hyperbolic structure, satisfying:
\begin{itemize}
\item For $n$ even, the paring locus $P_{n}$ consists of two rank-1 cusps and the boundary has topological type $\partial M_{n}\cong\Sigma_{n-2,4}$.
\item For $n$ odd, the paring locus $P_{n}$ consists of three rank-1 cusps and the boundary has topological type $\partial M_{n}\cong\Sigma_{n-3,6}$.
\end{itemize} For each integer $n\ge2$, the skinning map $\sigma_{M_{n}}$ has a critical point.}
\label{cor}
\end{cor}

This work owes a debt of inspiration to numerical experiments developed by Dumas-Kent ~\cite{dumas-kent2} which suggest the presence of a critical point of the skinning map associated to the manifold in Theorem \ref{thm}. Their work also examines some other closely related manifolds and skinning maps, at least one of which appears to be a diffeomorphism onto its image. It would be interesting to have a unified understanding of these different behaviours.

The reader may note that the proof of Theorem \ref{thm} is only one of existence. It does not identify a critical point, determine the number of critical points, or determine any local degrees of $\sigma_{M}$ near critical points. However, for this skinning map there is some evidence of a unique simple critical point in the experiments of Dumas-Kent. Numerical tools developed by the author support this observation, and, curiously, suggest that the critical point occurs where $M$ has hexagonal convex core boundary (that is, as a point in moduli space the convex core boundary is isomorphic to the Poincar\'{e} metric on $\C\setminus\{0,1,e^{\pi i/3}\}$). This would seem to imply some geometrical significance to the critical point, which remains elusive.

In \S\ref{background}, we introduce background and notation. The proof that critical points of skinning maps persist under certain finite covers is found in \S\ref{finite covers}. In \S\ref{example} and \S\ref{param def path} we introduce the 3-manifold relevant to Theorem \ref{thm}, and a path of geometrically finite structures on it, noting that this path maintains some important symmetry. In \S\ref{4p} we study the implications of this symmetry on the geometry of a four-punctured sphere. We apply some of these implications in \S\ref{cc bound} to the convex core boundaries of our path. This information is used in \S\ref{non-injectivity} which collects the main ideas for the proof of Theorem \ref{thm}, deferring some computations to \S\ref{compute} and the Appendix. Finally, \S\ref{covers} presents a family of finite covers of the example from Theorem \ref{thm}, which proves Corollary \ref{cor}.

\section{Background and Notation}
\label{background}
Let $\Sigma=\Sigma_{g,n}$ be a smooth surface of genus $g$ with $n$ punctures.

Let $\mathcal{S}$ denote the set of nonperipheral unoriented simple closed curves on $\Sigma$ and $\ML(\Sigma)$ the space of measured geodesic laminations. Recall the natural inclusion $\R_{+}\times\mathcal{S}\into\ML(\Sigma)$, which has dense image. The quotient of $\ML(\Sigma)$ under the action of multiplication by positive weights will be written $\PML(\Sigma)$. See \cite{casson-bleiler} and \cite{flp} for details.

Recall that the {\it Teichm\"{u}ller space} of $\Sigma$, denoted $\T(\Sigma)$, is the set of marked complex structures (equivalently, marked hyperbolic structures, via uniformization) on $\Sigma$ up to marking equivalence. Recall that $\T(\Sigma)$ is naturally a complex manifold homeomorphic to $\C^{6g-6+2n}$. 

For a smooth manifold $N$, possibly with boundary, define $MCG^{*}(N):=\nicefrac{\mathrm{Diff}(N)}{\mathrm{Diff}_{0}(N)}$, the {\it extended mapping class group} of $N$. Pre-composition of a marking with a diffeomorphism of the surface $\Sigma$ descends to an action of $MCG^{*}(\Sigma)$ on $\T(\Sigma)$. Note that we allow orientation-reversing diffeomorphisms. When $\Sigma$ is the boundary of a 3-manifold $M$, there is a restriction map $r:MCG^{*}(M)\to MCG^{*}(\Sigma)$. See \cite{lehto}, \cite{farb-margalit}, and \cite{flp} for details about $\T(\Sigma)$ and $MCG^{*}(\Sigma)$.

A {\it Kleinian group} $\Gamma$ is a discrete subgroup of $\PSL(2,\C)$. The action of $\Gamma$ on $\CP^{1}$ has a maximal domain of discontinuity $\Omega_{\Gamma}$, and its complement is the {\it limit set}, $\Lambda_{\Gamma}$. One component of the complement of a totally geodesic plane in $\H^{3}$ is a {\it supporting half-space} for a connected component of $\Omega_{\Gamma}$---or, less specifically, a supporting half-space for $\Gamma$---if it meets the conformal boundary $\CP^{1}$ in that component, and the closure of the intersection with $\CP^{1}$ intersects $\Lambda_{\Gamma}$ in at least two points. The totally geodesic boundary of a supporting half-plane for $\Gamma$ is a {\it support plane} for $\Gamma$. The {\it convex hull} of $\Gamma$ is the complement of the union of all of its supporting half-spaces and the {\it convex core} of $\Gamma$ is the quotient of the convex hull by $\Gamma$. When an $\epsilon$-neighborhood of the convex core is finite volume, $\Gamma$ is {\it geometrically finite}. See \cite{maskit} and \cite{canepsgreen} for details.

Let $N$ be a topological space with finitely-generated fundamental group and let $G$ be an algebraic Lie group. Then $\Hom(\pi_{1}N,G)$ naturally has the structure of an algebraic variety. There is a conjugation action of $G$ on $\Hom(\pi_{1}N,G)$, and the quotient, interpreted in the sense of geometric invariant theory, is the {\it $G$-character variety of $N$}. We will denote the $G$-character variety of $N$ by $\X(N,G)$ or, when $G$ is understood to be $\PSL_{2}\C$, by $\X(N)$. Recall that there is a natural holomorphic structure that $\X(N)$ inherits from $\PSL_{2}\C$. See \cite{goldman1} and \cite{goldman} for details.

A Kleinian group $\Gamma$ is {\it Fuchsian} if it is conjugate, in $\PSL_{2}\C$, into $\PSL_{2}\R$, and it is {\it quasi-Fuchsian} if it is conjugate under a quasiconformal map into $\PSL_{2}\R$. We denote by $\QF(\Sigma)$ the locus of $\X(\Sigma)$ consisting of faithful representations with quasi-Fuchsian image without accidental parabolics. That is, for $[\rho]\in\QF(\Sigma)$, the image $\rho(\gamma)$ is parabolic if and only if $\gamma$ is homotopic to a boundary component of $\Sigma$. Recall Bers' Theorem \cite{bers},
\begin{thm}[Simultaneous Uniformization]{There exists an identificaiton $\mathrm{AB}_{\Sigma}:\QF(\Sigma)\cong\T(\Sigma)\times\T(\overline{\Sigma})$, a biholomorphism of complex manifolds.}
\label{simunif}
\end{thm}
When $\rho$ is faithful, we identify $[\rho]\in\X(N)$ with the conjugacy class of its image $\rho(\pi_{1}N)$ in $\PSL_{2}\C$. We write $Q$ for the map $\mathrm{AB}_{\Sigma}^{-1}$, so that, for $X,Y\in\T(\Sigma)$, we have $Q(X,\overline{Y})\in\QF(\Sigma)$.

A {\it pared 3-manifold} is a pair, $(H_{0},P)$, where $H_{0}$ is a compact 3-manifold with boundary and $P$ is a disjoint union of incompressible annuli and tori in $\partial H_{0}$, such that:
\begin{itemize}
\item $P$ contains all of the tori in $\partial H_{0}$
\item Any embedded cylinder \(\left(S^{1}\times I,S^{1}\times\partial I\right)\into(H_{0},P)\) is homotopic, relative to the boundary, into $P$
\end{itemize}
By the `boundary' of a pared 3-manifold, $\partial(H_{0},P)$, we mean $\partial H_{0}\setminus P$. We say that $(H_{0},P)$ is {\it acylindrical} if any embedded cylinder \(\left(S^{1}\times I,(S^{1}\times\partial I)\right)\into\left(H_{0},\partial (H_{0},P)\right)\) is homotopic, relative to the boundary, into $\partial(H_{0},P)$. See \cite[p.~434]{mcmullen1} for details.

When $M=(H_{0},P)$ is a pared 3-manifold, we denote by $\X(M)\subset\X(H_{0})$ the locus of points for which the image of curves homotopic into $P$ are parabolic. The subset $\GF(M)\subset\X(M)$ consists of $[\rho]\in\X(M)$ such that $\rho\left(\pi_{1}H_{0}\right)$ is a geometrically finite Kleinian group, and such that $\rho$ has no accidental parabolics. That is, the image $\rho(\gamma)$ is parabolic if and only if $\gamma$ is homotopic into $P$. Let $\mathrm{Diff}(H_{0},P)$ be the diffeomorphisms of $H_{0}$ that preserve the set $P$ and $MCG^{*}(M):=\nicefrac{\mathrm{Diff}(H_{0},P)}{\mathrm{Diff}_{0}(H_{0},P)}$. Note that $MCG^{*}(M)$ acts on $\X(M)$ preserving $\GF(M)$. 

Since $\QF(\Sigma)$ identifies with $\GF(\Sigma\times[0,1])$, it is natural to seek a generalization of Bers' Uniformization Theorem to 3-manifolds other than $\Sigma\times[0,1]$. Such a theorem is provided by the deformation theory developed by Ahlfors, Bers, Kra, Marden, Maskit, Sullivan, and others. We refer to this identification as the `Ahlfors-Bers' parameterization:

\begin{thm}[Ahlfors-Bers Parameterization]{For a pared 3-manifold $M=(H_{0},P)$ with incompressible boundary, there exists an identification $\mathrm{AB}_{M}:\GF(M)\cong\T(\partial M)$, a biholomorphism of complex manifolds.}
\end{thm}
\noindent Note that $\partial M$ may be disconnected, in which case $\T(\partial M)$ is the product of the Teichm\"{u}ller spaces of the components. 
\vspace{.4cm}

For the remainder of the section, let $M=(H_{0},P)$ be a pared 3-manifold, such that $\Sigma=\partial M$ is incompressible and connected. Fix a choice of basepoints for $M$ and $\Sigma$, and a path connecting them, so that inclusion induces the well-defined homomorphism $\iota_{*}:\pi_{1}\Sigma\into\pi_{1}M$ and the restriction morphism $\iota^{*}$ on representations.

Suppose that $\left[\hat{\Gamma}\right]\in\GF(M)$ has $\Omega_{\hat{\Gamma}}\ne\emptyset$, and let $\Gamma:=\iota^{*}\hat{\Gamma}$. The choices above fix an identification of universal covers $\widetilde{M}\cong\H^{3}$ and $\widetilde{\Sigma}\cong U_{0}$, for some component $U_{0}\subset\Omega_{\hat{\Gamma}}$. In this case, $\Gamma=\mathrm{Stab}_{\hat{\Gamma}}(U_{0})$ and \cite[Cor.~6.5.]{marden} shows that $\Gamma$ must be geometrically finite. Because $\Sigma$ is incompressible in $M$, the domain $U_{0}$ must be simply-connected. Because $\hat{\Gamma}$ is without accidental parabolics, $\Gamma$ is too. We may conclude:

\begin{lemma}{For $\left[\hat{\Gamma}\right]\in\GF(M)$ satisfying $\Omega_{\hat{\Gamma}}\ne\emptyset$, and $\Gamma=\iota^{*}\hat{\Gamma}$, we have $[\Gamma]\in\QF(\Sigma)$.}
\label{restrqF}
\end{lemma}

Note that the assumption $\Omega_{\hat{\Gamma}}\ne\emptyset$ rules out a fibered hyperbolic 3-manifold, where the Kleinian group corresponding to the fiber is geometrically infinite, and the conclusion of Lemma \ref{restrqF} fails. 

\begin{definition}
{For $\left[\hat{\Gamma}\right]\in\GF(M)$ satisfying $\Omega_{\hat{\Gamma}}\ne\emptyset$, and $\Gamma=\iota^{*}\hat{\Gamma}$, the domains $\Omega_{\Gamma}$ and $\Omega_{\hat{\Gamma}}$ share a connected component, which we refer to as the {\it top} of $\Gamma$. By Lemma \ref{restrqF}, the domain $\Omega_{\Gamma}$ has one other component, which we refer to as the {\it bottom}. When clear from context, we may refer to the quotient of the top component of $\Omega_{\Gamma}$ as the top of $\Gamma$.}
\end{definition}

Lemma \ref{restrqF} makes Thurston's definition of the skinning map possible. Below, we assume the connectedness of $\Sigma$ (for $\Sigma$ disconnected, see \cite{otal1}).

\begin{definition}
{The skinning map $\sigma_{M}:\T(\Sigma)\to\T(\overline{\Sigma})$ fits into the following commutative diagram:
$$\xymatrix{
\T(\Sigma)\ar[drrrrr]_{\sigma_{M}}\ar[rr]^-{\mathrm{AB}_{M}} && \GF(M)\ar[r]^{\iota^{*}} & \QF(\Sigma)\ar[rr]^-{\mathrm{AB}_{\Sigma}} && \T(\Sigma)\times\T(\overline{\Sigma})\ar[d]^{proj_{2}} \\
 & & &&& \T(\overline{\Sigma})
}$$
In words, $\iota^{*}\circ\mathrm{AB}_{M}:\T(\Sigma)\to\QF(\Sigma)$ associates to $X\in\T(\Sigma)$ a quasi-Fuchsian structure $Q(X,\overline{Y})\in\QF(\Sigma)$, where the top of $Q(X,\overline{Y})$ is $X$, and the bottom is $\overline{Y}$. The skinning map is given by $\sigma_{M}(X)=\overline{Y}$.}\end{definition}

Because each of the maps above is holomorphic, $\sigma_{M}$ is holomorphic. As a result of the naturality of the Ahlfors-Bers identifications $\sigma_{M}$ is equivariant for the actions of $MCG^{*}(M)$, $MCG^{*}(\Sigma)$, and the restriction map $r:MCG^{*}(M)\to MCG^{*}(\Sigma)$. As a consequence,

\begin{prop}{ For each $\phi\in MCG^{*}(M)$, the fixed point set \;$\mathrm{Fix}\;r(\phi)$ is preserved by the skinning map, i.e.~$\sigma_{M}(\mathrm{Fix}\;r(\phi))\subset\mathrm{Fix}\;r(\phi)$.}
\label{prop fix}
\end{prop}
Note that the two copies of $\mathrm{Fix}\;r(\phi)$ in Proposition \ref{prop fix} lie in Teichm\"{u}ller spaces of surfaces with opposite orientations. Because there is a canonical anti-holomorphic isometry $\overline{\cdot}:\T(\overline{\Sigma})\to\T(\Sigma)$, we may sometimes elide this difference and view the image of the skinning map in $\T(\Sigma)$. 

The reader should note that the existence of any nontrivial mapping class of the 3-manifold $M$ is not assured. In the 3-manifold from Theorem \ref{thm}, this existence is an essential tool in our analysis, allowing a reduction of dimension.

\section{Skinning Maps of Finite Covers of $M$}
\label{finite covers}

Let $p:(M',\Sigma')\to(M,\Sigma)$ be a finite covering of manifolds with boundary. We will denote the restriction $p|_{\Sigma'}$ by $p$ as well.

\begin{prop}{Suppose that $\sigma_{M}$ has a critical point at $X\in\T(\Sigma)$. Then $\sigma_{M'}$ has a critical point at $p^{*}X\in\T(\Sigma')$.}
\label{cover}
\end{prop}
\begin{proof}
{Denote the image $\mathcal{E}_{M}:=\iota^{*}\GF(M)\subset\QF(\Sigma)$, and the {\it Bers slice} through $Y\in\T(\Sigma)$ by $\mathcal{B}_{Y}:=\{Q(X,\overline{Y})\;|\;X\in\T(\Sigma)\}\subset\QF(\Sigma)$ (that is, quasi-Fuchsian manifolds with $\overline{Y}\in\T(\overline{\Sigma})$ on the bottom). Note that $\overline{Y}=\sigma_{M}(X)$ if and only if $\mathcal{E}_{M}$ and $\mathcal{B}_{Y}$ intersect at $Q(X,\overline{Y})$. If $\overline{Y}=\sigma_{M}(X)$, then $\sigma_{M}$ has a critical point at $X$ if and only if $\mathcal{E}_{M}$ and $\mathcal{B}_{Y}$ have a tangency at $Q(X,\overline{Y})$.

Note that $\Sigma'$ is incompressible in $M'$, since a compressing curve in $\Sigma'$ would push-forward under $p$ to a curve on $\Sigma$ in the kernel of $\pi_{1}\Sigma\into\pi_{1}M$. We may thus consider the following diagram:
\begin{equation}
\label{diagram:commute:covers}
\xymatrix{
\T(\Sigma') \ar[rr]^{\mathrm{AB}_{M'}} && \GF(M') \ar[r]^{\iota^{*}} & \QF(\Sigma') \ar[rr]^{\mathrm{AB}_{\Sigma'}} && \T(\Sigma')\times\T(\overline{\Sigma}')\\
\T(\Sigma) \ar[u]^{p^{*}}\ar[rr]_{\mathrm{AB}_{M}} && \GF(M) \ar[u]^{p^{*}}\ar[r]_{\iota^{*}} & \QF(\Sigma) \ar[u]^{p^{*}}\ar[rr]_{\mathrm{AB}_{\Sigma}} && \T(\Sigma)\times\T(\overline{\Sigma}) \ar[u]^{p^{*}}
}\end{equation}
The maps $p^{*}$ above each correspond to lifting structures via $p$. For example, the natural map induced by restriction $\GF(M)\to\GF(M')$ can be interpreted as lifting hyperbolic structures via $p$.

In the center of the diagram, each map is given by restricting a representation to a subgroup. This is induced by the commutative diagram of subgroups,
\[\xymatrix{
\pi_{1}\Sigma'\ar@{}[r]|-*[@]{\subset} & \pi_{1}M'\\
\pi_{1}\Sigma\ar@{}[u]|-*[@]{\cap}\ar@{}[r]|-*[@]{\subset} & \pi_{1}M \ar@{}[u]|-*[@]{\cap}
}\]
and hence the center of diagram (\ref{diagram:commute:covers}) commutes.

The outside pieces of diagram (\ref{diagram:commute:covers}) commute because the map $\mathrm{AB}_{M}$ is natural with respect to passing to finite covers: Given a point $[\Gamma]\in\GF(M)$, let $\Gamma':=p^{*}(\Gamma)$, let $\mathrm{AB}_{M}^{-1}[\Gamma]=X$, and let $X':=\mathrm{AB}_{M'}^{-1}[\Gamma']$. Because $\Gamma'<\Gamma$ is finite-index, the tops of the domains of discontinuity of $\Gamma$ and $\Gamma'$ are the same set. Thus $X'$ holomorphically covers $X$, with compatible markings so that $p^{*}(X)=X'$. 

Consider a critical point $X\in\T(\Sigma)$, with $\overline{Y}=\sigma_{M}(X)$, i.e.~a tangency between $\mathcal{B}_{Y}$ and $\mathcal{E}_{M}$ in $\QF(\Sigma)$. By commutativity of the diagram, $p^{*}\mathcal{B}_{Y}\subset\mathcal{B}_{Y'}$, where $Y'=p^{*}(Y)$, and $p^{*}\mathcal{E}_{M}\subset\mathcal{E}_{M'}$. Thus, in order to see a critical point of $\sigma_{M'}$ at $p^{*}X$, it is enough to observe that $p^{*}:\QF(\Sigma)\to\QF(\Sigma')$ is an immersion: A tangency between $\mathcal{B}_{Y}$ and $\mathcal{E}_{M}$ at $Q(X,\overline{Y})$ lifts to a tangency between $\mathcal{B}_{Y'}$ and $\mathcal{E}_{M'}$ at $p^{*}Q(X,\overline{Y})=Q(p^{*}X,p^{*}\overline{Y})$.

In order to see that $p^{*}:\QF(\Sigma)\to\QF(\Sigma')$ is an immersion, we must consider the smooth structure on $\X(\Sigma)$. One interpretation of this structure identifies $T_{[\rho]}\X(\Sigma)$ with $H^{1}\left(\Sigma,(\mathrm{sl}_{2}\C)_{\rho}\right)$, the vector space of 1-forms with values in the flat $\mathrm{sl}_{2}\C$-bundle on $\Sigma$ associated to $\rho$. (See \cite{goldman1} for details about the infinitesimal deformation theory of $\X(\Sigma)$). The restriction map $p^{*}:\X(\Sigma)\to\X(\Sigma')$ naturally induces the pullback map $p^{*}:H^{1}\left(\Sigma,(\mathrm{sl}_{2}\C)_{\rho}\right)\to H^{1}\left(\Sigma',(\mathrm{sl}_{2}\C)_{p^{*}\rho}\right)$ on tangent spaces.

Finite covering maps induce injective pullback maps on cohomology groups: If the pullback $p^{*}\phi$ is a coboundary $df$, one may average $f$ over the finite sheets of the cover to obtain a form that descends, showing that $\phi$ was also a coboundary. Hence $d_{[\rho]}p^{*}:T_{[\rho]}\X(\Sigma)\to T_{p^{*}[\rho]}\X(\Sigma')$ is injective, $p^{*}:\X(\Sigma')\to\X(\Sigma)$ is an immersion, and $\sigma_{M'}$ has a critical point at $p^{*}X$.}
\end{proof} 

\section{The Example 3-manifold}
\label{example}
\begin{figure}[h]
	\begin{minipage}[]{.26\linewidth}
		\centering
		\includegraphics[height=9.cm]{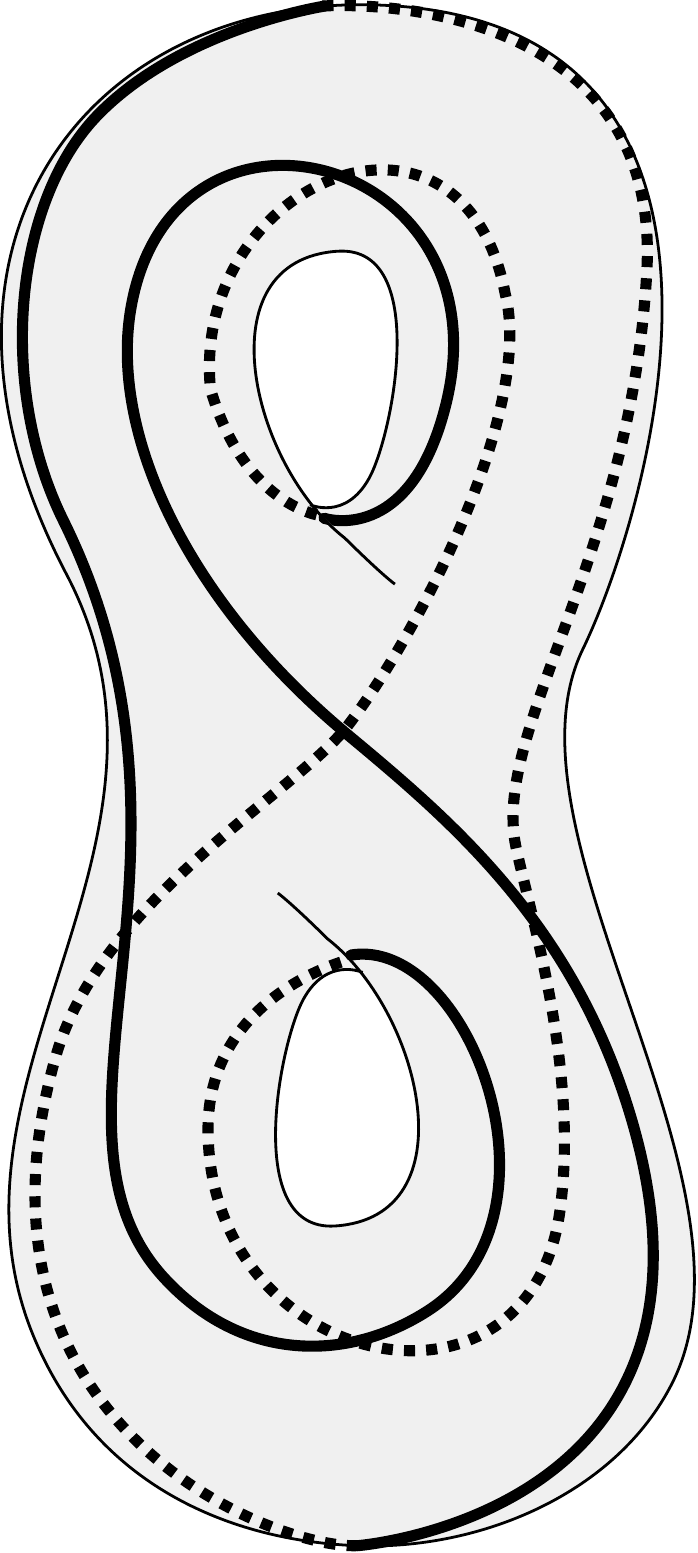}
		\caption{Core curves of annuli in $P$}
		\label{MP1}
	\end{minipage}
	\hspace{1cm}
	\begin{minipage}[]{.26\linewidth}
		\centering
		\begin{lpic}[clean]{MP1send2(,9.cm)}
			\Large
			\lbl[bl]{55,138;$A$}
			\lbl[bl]{14,18;$B$}
			\lbl[]{52,88;$x$}
		\end{lpic}
		\caption{Generators for $\pi_{1}(H_{2},x)$}
		\label{MP2}
	\end{minipage}
	\hspace{1cm}	
	\begin{minipage}[]{.26\linewidth}
		\centering
		\begin{lpic}[clean]{MP3send2(,9.cm)}
			\Large
			\lbl[]{41,73;$\delta_{1}$}
			\lbl[]{69,62;$\delta_{2}$}
			\lbl[]{56,105;$\delta_{3}$}
			\lbl[]{07,30;$\delta_{4}$}
		\end{lpic}
		\caption{Generators for $\pi_{1}(\Sigma,x)$}
		\label{MP3}
	\end{minipage}
\end{figure}

Consider the following pared 3-manifold $M=(H_{2},P)$: Let $H_{2}$ denote the closed genus-2 handlebody and $P$ the union of the annuli obtained as regular neighborhoods of the curves pictured in Figure \ref{MP1}. Fix a choice of basepoint $x\in\partial H_{2}\setminus\overline{P}$, and the presentation $\pi_{1}(H_{2},x)=\langle A,B\rangle$ (see Figure \ref{MP2}). For a natural choice of basis for $\pi_{1}\left(\Sigma_{2,0},x\right)=\langle a_{1},b_{1},a_{2},b_{2}\;|\;[a_{1},b_{1}][a_{2},b_{2}]=1\rangle$ the conjugacy classes of the core curves of $P$ are $\{\left[b_{1}a_{1}b_{1}b_{2}\left[a_{1},b_{1}\right]\right],\left[b_{2}a_{2}b_{2}b_{1}\right]\}$\footnote{The reader is warned of the notational offense that the choices above force $\pi_{1}\left(\Sigma_{2,0},x\right)\to\pi_{1}(M,x)$ to be given by $b_{1}\mapsto A$ and $b_{2}\mapsto B$. Though inconvenient, $a_{i}$ and $b_{i}$ will play no further role in our analysis, and the reader may ignore $\pi_{1}\left(\Sigma_{2,0},x\right)$.}.

Since the annuli in $P$ are {\it disk-busting} in the genus-2 surface, \cite[Lem.~1.15.]{otal2} guarantees that the boundary is incompressible and $(H_{2},P)$ is acylindrical. Since the annuli in $P$ are non-separating, $\Sigma=\partial M$ is a 4-holed sphere. Fix the presentation $\pi_{1}(\Sigma,x)=\langle \delta_{1},\delta_{2},\delta_{3},\delta_{4}\;|\;\prod\delta_{i}=1\rangle$, with the $\delta_{i}$ as pictured in Figure \ref{MP3}: $\delta_{i}$ takes the path drawn to a boundary component and winds around it counter-clockwise. Deleting $P$ from $\partial H_{2}$ we arrive at a topological picture of $\Sigma$, shown in Figure \ref{4holes}.

\begin{figure}[h]
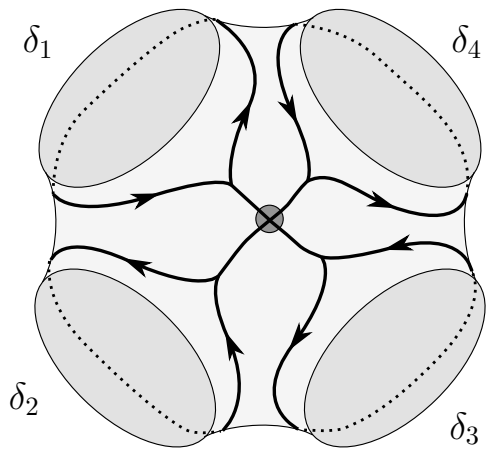

	\begin{lpic}[clean]{4holespheresend(6cm)}
		\Large
		\lbl[]{2,205;$\delta_{1}$}
		\lbl[]{218,205;$\delta_{4}$}
		\lbl[]{-5,25;$\delta_{2}$}
		\lbl[]{217,10;$\delta_{3}$}
	\end{lpic}
	\caption{Another view of $\pi_{1}(\Sigma,x)$}
	\label{4holes}
\end{figure}

Choosing the constant path from $x$ to $x$, we record the homomorphism induced by inclusion $\iota_{*}:\pi_{1}(\Sigma,x)\into\pi_{1}(M,x)$ in the chosen generators:

\begin{itemize}
\item $\iota_{*}(\delta_{1})=A^{-2}B$
\item $\iota_{*}(\delta_{2})=B^{-2}A^{-1}$
\item $\iota_{*}(\delta_{3})=AB^{-1}A=A\;\iota_{*}(\delta_{1})^{-1}\;A^{-1}$
\item $\iota_{*}(\delta_{4})=A^{-1}B^{2}A^{2}=A^{-2}\;\iota_{*}(\delta_{2})^{-1}\;A^{2}$
\end{itemize}

We suppress the basepoint $x$ and the notation $\iota_{*}$ in what follows, and view $\pi_{1}(\Sigma)$ as a subgroup of $\pi_{1}(M)$. In everything that follows, we fix the notation $M=(H_{2},P)$ and $\Sigma$ for the 4-holed sphere boundary. We now describe a geometrically finite hyperbolic structure on $M$. Details about similar structures (gluings of regular right-angled ideal polyhedra) can be found in \cite{chesebro-deblois}.

For $v,w\in\left(\CP^{1}\right)^{3}$, each with pairwise distinct entries, let $m(v,w)$ be the unique M\"{o}bius transformation taking the triple $v$ to the triple $w$. Fix an identification $\partial_{\infty}\H^{3}\cong\CP^{1}$, and let $\O$ be the ideal octahedron in $\H^{3}$ with totally geodesic triangular faces whose vertices are $\{1,0,-1,i,-i,\infty\}$. 

Consider the following representation of $\pi_{1}H_{2}$:
\begin{align*}
\tilde{\rho}_{1}:\langle A,B\rangle \;&\to\hspace{.2cm}\SL_{2}\C\\
A\;\;\;\;&\mapsto \;\;m((-1,i,0),(0,-i,1))\;\;\;\;\hspace{.05cm}=\;\frac{1+i}{2}\pmat{1}{1}{1+2i}{1}\\
B\;\;\;\;&\mapsto \;\;m((-1,-i,\infty),(1,\infty,i))\;=\;-\frac{1+i}{2}\pmat{i}{-1+2i}{1}{i}
\end{align*}
Let $\rho_{1}:\langle A,B\rangle\to\PSL_{2}\C$ be the representation induced by $\tilde{\rho}_{1}$.

\begin{figure}[]
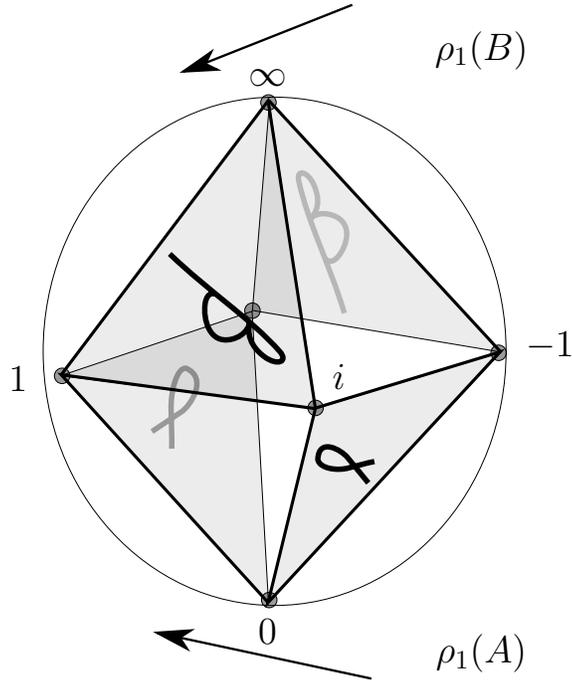

	\centering
		\begin{lpic}[clean]{octahedronsend(,9.cm)}
			\Large
			\lbl[]{160,230;$\rho_{1}(B)$}
			\lbl[]{82,220;$\infty$}
			\lbl[]{-9,110;$1$}
			\lbl[]{108,111;$i$}
			\lbl[]{185,122;$-1$}
			\lbl[]{82,18;$0$}
			\lbl[]{160,10;$\rho_{1}(A)$}
		\end{lpic}
		\caption{The octahedron $\O$ with side identifications labelled by $\alpha$ and $\beta$.}
		\label{octpic}
\end{figure} 
Geometrically, $\rho_{1}(A)$ and $\rho_{1}(B)$ perform face identifications on $\O$ with hyperbolic isometries (see Figure \ref{octpic}), and so $\rho_{1}$ is faithful by a standard Ping-Pong Lemma argument. Let $\hat{\Gamma}_{1}=\rho_{1}\langle A,B\rangle$. By the Poincar\'{e} Polyhedron Theorem \cite[p.~75]{maskit}, $\hat{\Gamma}$ is discrete, and $\nicefrac{\H^{3}}{\hat{\Gamma}}$ is a hyperbolic structure on the genus-2 handlebody. Since all of the dihedral angles of $\O$ are $\frac{\pi}{2}$, the non-paired faces meet flush. Thus the quotient $\nicefrac{\O}{\hat{\Gamma}_{1}}$ is a hyperbolic structure on the genus-2 handlebody, with totally geodesic boundary, homotopy equivalent to $\nicefrac{\H^{3}}{\hat{\Gamma}_{1}}$. Horoball neighborhoods of the six ideal vertices glue to two rank-1 cusp neighborhoods in the complete hyperbolic manifold $\nicefrac{\H^{3}}{\hat{\Gamma}_{1}}$. 

Denote the convex core of $\hat{\Gamma}_{1}$ by $\mathcal{C}$. Since $\O$ is the convex hull of points that are in the limit set, $\nicefrac{\O}{\hat{\Gamma}_{1}}\subset\mathcal{C}$. Since $\mathcal{C}$ is minimal among convex subsets of $\nicefrac{\H}{\hat{\Gamma}_{1}}$ for which inclusion is a homotopy equivalence, we see that $\mathcal{C}=\nicefrac{\O}{\hat{\Gamma}_{1}}$. Since $\O$ has finite volume and nonempty interior, $\hat{\Gamma}_{1}$ is geometrically finite.

The $\rho_{1}$-images of the core curves of $P$ are distinct maximal parabolic conjugacy classes in $\hat{\Gamma}_{1}$. Any parabolics in $\hat{\Gamma}_{1}$ must be conjugate to parabolics stabilizing the equivalence class of a vertex of $\O$ \cite[p.~131]{maskit}. Since $P$ consists of two annuli, and since the vertices split into two equivalence classes, $\rho_{1}$ determines a one-to-one correspondence between components of $P$ and conjugacy classes of maximal parabolic subgroups of $\hat{\Gamma}_{1}$. This means we are in the setting of \cite[Lem.~2.6.]{chesebro-deblois} and $\rho_{1}$ is induced by a homeomorphism from the convex core of $\hat{\Gamma}_{1}$, $\nicefrac{\O}{\hat{\Gamma}_{1}}$, to $H_{2}\setminus P$. We denote this point $[\rho_{1}]\in\GF(M)\subset\X(M)$. 

The octahedron has an evident self-map which respects the gluings: Reflect through an equatorial plane and rotate by $90^{\circ}$ around the axis perpindicular to the plane. In the coordinates chosen in Figure \ref{octpic}, this is the anti-M\"{o}bius map $z\mapsto i/\bar{z}$. This self-map of $\O$ thus descends to an orientation-reversing isometry of the hyperbolic 3-manifold $\nicefrac{\H^{3}}{\hat{\Gamma}_{1}}$ (cf. \S\ref{param def path}).

Before analyzing this symmetry of $(M,\Sigma)$ (cf. \S\ref{4p}), the octahedron $\O$ provides a natural way to deform the representation $\rho_{1}$.

\section{Deforming the example}
\label{param def path}

Consider now the path of representations:
\begin{align}
\tilde{\rho}_{t}:\langle A,B\rangle \;&\to\hspace{.2cm}\SL_{2}\C\notag\\
\label{rhodef}
A\;\;\;\;&\mapsto \;\;m((-1,it,0),(0,-it,1))\;\;\;\;\hspace{.05cm}=\;\frac{1}{t(t-i)}\pmat{t^{2}}{t^{2}}{1+2it}{t^{2}}\\
B\;\;\;\;&\mapsto \;\;m((-1,-it,\infty),(1,\infty,it))\;=\;-\frac{i}{t+i}\pmat{it}{-1+2it}{1}{it}\notag
\end{align}
Once again, let $\rho_{t}:\langle A,B\rangle\to\PSL_{2}\C$ be the representations induced by $\tilde{\rho}_{t}$. Taking $t=1$, we recover the representation $\rho_{1}$ from \S\ref{example}.

For each $t\in\R$, let $\hat{\Gamma}_{t}:=\rho_{t}\langle A,B\rangle$ and $\Gamma_{t}:=\rho_{t}(\pi_{1}\Sigma)$. 

\begin{lemma}
{The maps $\left[\rho_{\cdot}\right]:(0,\infty)\to\X(M)$ and $\left[\rho_{\cdot}|_{\pi_{1}(\Sigma)}\right]:(0,1]\to\X(\Sigma)$ are injective. For $t>0$, $[\rho_{t}|_{\pi_{1}(\Sigma)}]$ is in the real locus of $\X(\Sigma)$ if and only if $t=1$. }
\label{param check}
\end{lemma}

\begin{proof}{Let $A_{t}:=\rho_{t}(A)$ and $B_{t}:=\rho_{t}(B)$. Noting that the square of the trace is a well-defined function on $\PSL_{2}\C$, for $i\in\{1,2,3,4\}$ we have
\[tr^{2}\rho_{t}(\delta_{i})=4.\]
Thus $\left[\rho_{t}\right]\in\X(M)\subset\X(H_{0})$. 

In order to show injectivity of the paths $[\rho_{\cdot}]$ and $[\rho_{\cdot}|_{\pi_{1}\Sigma}]$ on $(0,\infty)$ and $(0,1)$, respectively, we show that $tr^{2}A_{t}$ and $tr^{2}\rho_{t}(\delta_{1}\delta_{2})$ are one-to-one functions of $t$ on the respective intervals.

We compute: 
$$tr^{2}A_{t}=\left(\frac{2t}{t-i}\right)^{2}$$
Note that $\frac{2t}{t-i}-1=\frac{t+i}{t-i}$. For $t>0$, the quantity $\frac{t+i}{t-i}$ is evidently a parameterization of the upper hemisphere of the unit circle centered at 0. Thus $\{tr^{2}A_{t}:t>0\}$ is a parameterization of the square of the upper hemisphere of the unit circle centered at 1. In particular, it is one-to-one on $(0,\infty)$.

For brevity, let $f(t):=tr\tilde{\rho}_{t}(\delta_{1}\delta_{2})$. A computation shows
\[f(t)=\frac{2t^{2}(t^{4}-22t^{2}-7)}{(1+t^{2})^{3}}+
i\frac{(t^{2}-1)(5t^{2}+1)^{2}}{t(1+t^{2})^{3}}.\]

For $t\in(0,1)$, the quantity $\Im f(t)<0$. Thus $f(t)^{2}$ is one-to-one on $(0,1)$ if and only if $f(t)$ is one-to-one on $(0,1)$. We compute
\[\frac{d}{dt}\left(\Im f(t)\right)=\frac{-(1+5t^{2})(5t^{6}-35t^{4}+7t^{2}-1)}{t^{2}(1+t^{2})^{4}}.\]
We now estimate, for $t\in(0,1)$,
\begin{align*}
5t^{6}-35t^{4}+7t^{2}-1&=&5&(t^{2}-1)^{3}\;\;-&20&(t^{2}-1)^{2}\;\;-&48&(t^{2}-1)\;\;-&64&\\
&<&0&\hspace{1cm}+&0&\hspace{1cm}+&48&\hspace{1cm}-&64&\\
&<&0
\end{align*}
Thus $\frac{d}{dt}\Im f(t)>0$, and $\Im f(t)$ is monotone increasing for $t\in(0,1)$. In particular, the function $tr^{2}\rho_{\cdot}(\delta_{1}\delta_{2}):(0,1)\to\C$ is one-to-one.

Finally, since the Kleinian group $\hat{\Gamma}_{1}$ has totally geodesic boundary, the boundary subgroup $[\Gamma_{1}]$ is in the real locus of $\X(\Sigma)$. As well, it is clear that $Im\;f(t)=0$, for $t>0$, if and only $t=1$. Thus $[\hat{\Gamma}_{t}]$ is in the real locus of $\X(\Sigma)$ if and only if $t=1$.}
\end{proof}

The `symmetry' of $\hat{\Gamma}_{1}$ (cf. \S\ref{example}) persists along the path $t\mapsto[\rho_{t}]$. In Lemma \ref{sym qf} we check that conjugation by the anti-M\"{o}bius map $\Psi_{t}(z)=it/\bar{z}$ descends to an isometry of $\nicefrac{\H^{3}}{\hat{\Gamma}_{t}}$. This isometry will be instrumental in our analysis of the path $[\rho_{t}]$ (cf. \S6 and \S7). 

\begin{lemma}{For all $t$, we have $\Psi_{t}\hat{\Gamma}_{t}\Psi_{t}^{-1}=\hat{\Gamma}_{t}$, and $\left[\Psi_{t}\Gamma_{t}\Psi_{t}^{-1}\right]=\left[\Gamma_{t}\right]$. Thus $\Psi_{t}$ induces a mapping class $\Psi\in MCG^{*}(M)$, with $\left[\hat{\Gamma}_{t}\right]\in\mathrm{Fix}\;\Psi$ and $[\Gamma_{t}]\in\mathrm{Fix}\;\Psi|_{\Sigma}$. The action of $\Psi|_{\Sigma}$ on the punctures is an order four permutation, and $\Psi|_{\Sigma}$ preserves the two homotopy classes of simple closed curves $\delta_{1}\delta_{3}$ and $\delta_{2}\delta_{4}$.}
\label{sym qf}
\end{lemma}

\begin{proof}{A calculation using the defintion of $\rho_{t}$ (equation (\ref{rhodef})) shows:\\
\begin{itemize}

\item $\Psi_{t}A_{t}\Psi_{t}^{-1}=B_{t}$
\item $\Psi_{t}B_{t}\Psi_{t}^{-1}=A_{t}^{-1}$

\end{itemize}\vspace{.5cm}
Thus $\Psi_{t}\hat{\Gamma}_{t}\Psi_{t}^{-1}=\hat{\Gamma}_{t}$, and $\Psi_{t}$ induces a mapping class $\Psi\in MCG^{*}(M)$ with $\left[\hat{\Gamma}_{t}\right]\in\mathrm{Fix}\;\Psi$. We will drop the restriction map notation and consider $\Psi$ as simultaneously an element of $MCG^{*}(M)$ and $MCG^{*}(\Sigma)$. We compute:

\vspace{.2cm}
\begin{itemize}

\item $\Psi_{t}\rho_{t}(\delta_{1})\Psi_{t}^{-1}=A_{t}^{2}\cdot\rho_{t}\left(\delta_{4}^{-1}\right)\cdot A_{t}^{-2}$

\item $\Psi_{t}\rho_{t}(\delta_{2})\Psi_{t}^{-1}=A_{t}^{2}\cdot\rho_{t}\left(\delta_{1}^{-1}\right)\cdot A_{t}^{-2}$

\item $\Psi_{t}\rho_{t}(\delta_{3})\Psi_{t}^{-1}=A_{t}^{2}\cdot\rho_{t}\left(\delta_{1}\delta_{3}\delta_{4}\right)\cdot A_{t}^{-2}$

\item $\Psi_{t}\rho_{t}(\delta_{4})\Psi_{t}^{-1}=A_{t}^{2}\cdot\rho_{t}\left(\delta_{1}\delta_{2}\delta_{4}\right)\cdot A_{t}^{-2}$

\end{itemize}\vspace{.2cm}
Thus $\Psi_{t}\Gamma_{t}\Psi_{t}^{-1}$ is conjugate to $\Gamma_{t}$, and $\Psi$ is an orientation-reversing mapping class of $\Sigma$ with $[\Gamma_{t}]\in\mathrm{Fix}\;\Psi$. 

Because $\delta_{1}\delta_{3}\delta_{4}=\delta_{1}\delta_{2}^{-1}\delta_{1}^{-1}\sim\delta_{2}^{-1}$, (and similarly $\delta_{1}\delta_{2}\delta_{4}\sim\delta_{3}^{-1}$), we see that the action of $\Psi$ on the conjugacy classes of punctures is a cyclic permutation. The fact that $\Psi$ preserves the geodesic representatives of the simple closed curves $\delta_{1}\delta_{3}$ and $\delta_{2}\delta_{4}$ is now immediate.}
\end{proof}

Note that, by Proposition \ref{prop fix}, the fixed set of the mapping class $\Psi$ is perserved by the skinning map. Namely:
\begin{align}
\label{prop fix ex}
\sigma_{M}\left(\mathrm{Fix}\;\Psi|_{\Sigma}\right)\subset\mathrm{Fix}\;\Psi|_{\Sigma}
\end{align}

Since $\GF(M)\subset\X(M)$ is open (see \cite[Thm.~10.1.]{marden}), there is some open interval about $1$ so that $\left[\rho_{t}\right]\in\GF(M)$ and, by Lemma \ref{restrqF}, $\left[\rho_{t}|_{\pi_{1}(\Sigma)}\right]\in\QF(\Sigma)$. Denote the maximal such open interval around $1$ by $U$. 

The path of quasi-Fuchsian groups $t\mapsto\Gamma_{t}$, for $t\in U$, induces two paths in $\T(\Sigma)$ corresponding to the top and bottom of $\Gamma_{t}$. By definition, $\sigma_{M}$ sends the top path to the bottom path. For $t\in U$, fix notation $[\Gamma_{t}]=Q(X_{t},Z_{t})$, so that $\sigma_{M}(X_{t})=Z_{t}$ and $X_{t},Z_{t}\in\mathrm{Fix}\;\Psi|_{\Sigma}$. Lemma \ref{param check} implies, in particular, that $\left\{X_{t}\;|\;t\in U\right\}$ is an injective path in $\T(\Sigma)$. In fact, the remainder of the paper is devoted to checking that $\left\{Z_{t}\;|\;t\in U\right\}$ is a non-injective path whose image is confined to the real one-dimensional submanifold $\mathrm{Fix}\;\Psi|_{\Sigma}\subset\T(\Sigma)$. Our strategy will be to examine the convex core boundary surfaces and their bending laminations. First, we need to better understand the set $\mathrm{Fix}\;\Psi|_{\Sigma}$. This is the subject of \S\ref{4p}.

\section{Rhombic 4-Punctured Spheres}
\label{4p}
In this section, we examine the symmetries of 4-punctured spheres, and collect some useful facts about a special symmetrical set in $\T(\Sigma)$. Let $Aut^{*}(X)$ denote the group of conformal or anti-conformal automorphisms of $X$. Recall that $\T(\Sigma)\cong\H$: The moduli space of complex structures on 4-punctured spheres, with ordered punctures, is $\C\setminus\{0,1\}$, and the universal cover of this complex manifold is $\H$.

It is well-known that for all $X\in\T(\Sigma)$ there exists a Klein 4-group of conformal automorphisms that acts trivially on $\mathcal{S}$, whose non-trivial elements are involutions that exchange punctures in pairs. Lemma \ref{char rhomb} characterizes a more restrictive orientation-reversing symmetry of a subset of $\T(\Sigma)$.  

\begin{lemma}{Let $X\in\T(\Sigma)$, $\xi,\eta\in\mathcal{S}$. The following are equivalent:
\begin{enumerate}

\item{$X$ may be written as the complement of the vertices of a Euclidean rhombus in $\CP^{1}$, with $\xi$ and $\eta$ as pictured (see Figure \ref{rhombpic1}).}

\item{In Fenchel-Nielsen coordinates, $\T(\Sigma)\cong\{(\ell,\theta)\;|\;\ell\in\R^{+},\;\theta\in\R\}$, the $\{\xi,\eta\}$-rhombic set is given by $\Rh_{\{\xi,\eta\}}\cong\{\theta=\pi/2\}$. The pants decomposition should be chosen with waistcurve $\xi$ and a certain choice of transverse curve $\eta^{*}$ corresponding to the pair $\{\xi,\eta\}$ (see Figure \ref{rhombpic2}).}

\item{There exists an orientation-reversing $\psi\in MCG^{*}(\Sigma)$ that acts as an order four cyclic permutation of the punctures, preserves $\xi$ and $\eta$, and $X\in\mathrm{Fix}\;\psi$.}

\end{enumerate}}
\label{char rhomb}
\end{lemma}
When $\xi$ and $\eta$ are clear from the context, we refer to the coordinate pictured in Figure \ref{rhombpic1} as the `rhombic coordinate'.
\begin{figure}[h]
	\begin{minipage}[]{.4\textwidth}
	\centering
	\begin{lpic}[clean]{4puncsend(6cm)}
				\small
				\lbl[]{67,119;$is$}
				\lbl[]{84,36;$-is$}
				\lbl[]{137,85;$1$}
				\lbl[]{13,69;$-1$}
		\Large
		\lbl[bl]{92,137;$\xi$}
		\lbl[]{135,40;$\eta$}	
	\end{lpic}
		\caption{The `rhombic coordinate' on a $\{\xi,\eta\}$-rhombic 4-punctured sphere}
		\label{rhombpic1}
	\end{minipage}
	\hspace{1cm}
	\begin{minipage}[]{.4\textwidth}
		\centering
	\begin{lpic}[clean]{4punceta-send(6cm)}
				\Large
				\lbl[]{100,10;$\eta^{*}$}	
	\end{lpic}
		\caption{The transversal $\eta^{*}$}
		\label{rhombpic2}
	\end{minipage}
\end{figure}

\begin{proof}{$(1)\Rightarrow(2)$: Choose waistcurve $\xi$, so that $X$ consists of two twice-punctured disks pasted together along $\xi$. Consider the following two geodesics, given in the rhombic coordinate:
\begin{itemize}
\item $\left\{\frac{is}{1-2t}:t\in[0,1]\right\}$, between $is$ and $-is$, through $\infty$
\item $\{1-2t:t\in[0,1]\}$, between $1$ and $-1$, through 0 
\end{itemize}
Note that these are geodesic by the presence of an isometry fixing each pointwise. We choose transversal $\eta^{*}$ as follows: The curve $\xi$ cuts each of these geodesics into one compact and two non-compact pieces. Join each of the compact pieces by turning right at each intersection with $\xi$.

In order to measure the twisting coordinate relative to the choice of transversal $\eta^{*}$, we observe that the four points of intersection of $\xi$ with the real and imaginary axes divide $\xi$ into four arcs. These arcs are cyclically permuted by the isometry $z\mapsto\frac{is}{\bar{z}}$, and hence are of equal length. This shows that the twisting coordinate is in $\Z\cdot\frac{\pi}{2}$, and the choice of $\eta^{*}$ guarantees the twisting coordinate is $\frac{\pi}{2}$.  

$(2)\Rightarrow(3)$: There is a reflection/twist symmetry of $X$ that preserves $\xi$ and $\eta$ and acts on punctures as desired.

$(3)\Rightarrow(1)$: By the existence of holomorphic involutions exchanging the punctures in pairs, there exists $\alpha\in Aut^{*}(X)$ such that $\phi:=\alpha\circ\psi$ is a simple transposition of the punctures that preserves $\xi$ and $\eta$. Without loss of generality, suppose $\phi$ interchanges the punctures enclosed by $\eta$. With an appropriate M\"{o}bius transformation, take these to $1$ and $-1$. Anti-conformal involutions of $\CP^{1}$ are involutions through a circle, and their fixed point sets are the circle they involve through. Since $\phi$ fixes the other two punctures, they lie on the fixed circle for $\phi$. We may thus apply another M\"{o}bius transformation, which fixes $1$ and $-1$ and takes this fixed circle to the imaginary axis. Apply an elliptic M\"{o}bius transformation fixing $1$ and $-1$ and centering the imaginary punctures about $0$, and $X$ has the desired form.}
\end{proof}

\begin{definition}{For $X\in\T(\Sigma)$, if there exists an order 4 orientation-reversing $\psi\in MCG^{*}(\Sigma)$ such that $X\in\mathrm{Fix}\;\psi$ and $\psi$ preserves the homotopy classes of simple closed curves $\xi$ and $\eta$, then X is {\it $\{\xi,\eta\}$-rhombic}.}
\end{definition}
\vspace{.1cm}
Let $\Rh_{\{\xi,\eta\}}:=\{X\in\T(\Sigma)\;|\;X\text{ is }\{\xi,\eta\}\text{-rhombic}\}$. We refer to the defining symmetry $\psi$ as {\it $\{\xi,\eta\}$-rhombic} as well. There is a natural choice of coordinate for the identification $\T(\Sigma)\cong\H$ which first identifies $\T(\Sigma_{0,4})\cong\T(\Sigma_{1,1})$. Relative to this choice, the reader may identify $\Rh_{\{\xi,\eta\}}$ as the line $\{\Re z=\frac{1}{2}\}$.

We look to the action of $MCG^{*}(\Sigma)$ on $\PML(\Sigma)$. As it turns out, the fixed points of the action of $\{\xi,\eta\}$-rhombic isometries are easy to characterize.

\begin{lemma}{For a $\{\xi,\eta\}$-rhombic mapping class $\psi\in MCG^{*}(\Sigma)$, the action of $\psi$ on $\PML(\Sigma)$ has $\mathrm{Fix}\;\psi=\{\xi,\eta\}$.}
\label{fix lam}
\end{lemma}

Because it is well-known, we do not include a detailed proof of the following fact: There is an identification $(\PML(\Sigma),\mathcal{S})\cong(\RP^{1},\QP^{1})$ which is equivariant with respect to a homomorphism $MCG^{*}(\Sigma)\to\PGL_{2}\Z$ (cf. \cite[p.~60]{farb-margalit}). Roughly speaking, for the four-punctured sphere there is a $MCG^{*}(\Sigma)$-equivariant map from rays in $\ML(\Sigma)$ to lines in $H_{1}(\Sigma,\R)$, which sends rays of laminations supported on simple closed curves to lines in $H_{1}(\Sigma,\mathbb{Q})\subset H_{1}(\Sigma,\R)$. The image of this map in a projectivized $H_{1}(\Sigma,\R)$ can be identified with slopes of lines exiting a chosen point on $\Sigma$, where the simple closed curves are identified with rational slopes. Because the action of an element of $MCG^{*}(\Sigma)$ on $\PML(\Sigma)$ is given by the action of an element of $\PGL_{2}\Z$ on $\RP^{1}$, each element of $MCG^{*}(\Sigma)$ which acts non-trivially on $\PML(\Sigma)$ has at most two fixed points. Lemma \ref{fix lam} follows.

For a curve $\gamma\in\mathcal{S}$, recall the real-valued functions $\ell_{\gamma}$ and $Ext_{\gamma}$ on $\T(\Sigma)$: For $X\in\T(\Sigma)$, the quantity $\ell_{\gamma}(X)$ is the hyperbolic length of the geodesic representative of $\gamma$ and the quantity $Ext_{\gamma}(X)$ is the extremal length of the family of curves homotopic to $\gamma$. See \cite{farb-margalit} and \cite{ahlfors} for details.

\begin{lemma}{The maps $\ell_{\xi},\;\ell_{\eta},\;Ext_{\xi},$ and $Ext_{\eta}$ each provide a diffeomorphism from $\Rh_{\{\xi,\eta\}}$ to $\R^{+}$.}
\label{param rhomb}
\end{lemma}

\begin{proof}{Fix pants decomposition with waistcurve $\xi$ and transverse curves $\eta^{*}$. By Lemma \ref{char rhomb}, $\ell_{\xi}|_{\Rh_{\{\xi,\eta\}}}$ is injective. Since one may construct $X\in\Rh_{\{\xi,\eta\}}$ with waistcurve of hyperbolic length as desired, and since length functions are smooth, the lemma is clear for $\ell_{\xi}$. The proof for $\ell_{\eta}$ is the same.

Along a Teichm\"{u}ller geodesic, the extremal lengths of the vertical and horizontal foliations each provide diffeomorphisms to $\R^{+}$ (this follows from \cite[Lem.~5.1]{gardiner-masur}, Gardiner's formula \cite[Thm.~8]{gardiner}, and the inverse function theorem). Thus it is enough to check:
\begin{enumerate}{
\item The set $\Rh_{\{\xi,\eta\}}$ is a Teichm\"{u}ller geodesic.
\item The projective classes of its foliations are given by $[\xi],[\eta]\in\PML(\Sigma)$.
}\end{enumerate}

The first follows from the fact that fixed point sets of isometries, in uniquely geodesic spaces such as $\left(\T(\Sigma),d_{\T}\right)$, are convex, while the second follows because the two foliations must be preserved by the rhombic symmetry, and Lemma \ref{fix lam} applies.}\end{proof}

\section{The Convex Core Boundaries of $\Gamma_{t}$}
\label{cc bound}
We return to the goal of understanding $\Gamma_{t}$ via its convex core boundary. Recall from Lemma \ref{sym qf} that $\Psi$ is an order-4 orientation-reversing isometry of $\Gamma_{t}$, cyclically permuting the punctures and preserving the simple closed curves $\delta_{1}\delta_{3}$ and $\delta_{2}\delta_{4}$. From here on we fix the the notation $\xi:=\delta_{1}\delta_{3}$ and $\eta:=\delta_{2}\delta_{4}$. The reader may notice that the isometry $\Psi$ is $\{\xi,\eta\}$-rhombic. Equation (\ref{prop fix ex}) now becomes
\begin{equation}
\label{preserve Rh}
\sigma_{M}\left(\Rh_{\{\xi,\eta\}}\right)\subset\Rh_{\{\xi,\eta\}}.
\end{equation}

Below, we present the computation of the bending angles. We choose the branch $[0,\pi]$ for $\cos^{-1}$ and $[0,2\pi)$ for $\arg$. Note that the existence of a `top' and `bottom' of $\Gamma_{t}$, relative to $\hat{\Gamma}_{t}$, implies that there is a `top' and `bottom' bending lamination. The computation of the bending measure for the top lamination is computable by the same methods presented as for the bottom.

\begin{lemma}{For $t\in U$, the top and bottom surfaces of the convex core boundary of $\Gamma_{t}$ are both in $\Rh_{\{\xi,\eta\}}$. The bending lamination on the bottom
of the convex core boundary is $\theta(t)\cdot\xi$, for 
\begin{equation}
\label{theta}
\theta(t)=\cos^{-1}\left(\frac{2t^{3}(3-t^{2})}{(1+t^{2})^{2}}\right).
\end{equation}}
\label{bend cc}
\end{lemma}

\begin{proof}{By Lemma \ref{sym qf}, $\Psi_{t}$ is a $\{\xi,\eta\}$-rhombic isometry of $\Gamma_{t}$. It thus sends the convex core to itself, so preserves the convex core boundary pleated surfaces, which are hence in $\Rh_{\{\xi,\eta\}}$. Moreover, $\Psi_{t}$ preserves the bending laminations $\{\lambda^{+},\lambda^{-}\}$, where $\lambda^{+}$ is on the top and $\lambda^{-}$ is on the bottom. By Lemma \ref{fix lam}, $\{[\lambda^{+}],[\lambda^{-}]\}\subset\{\xi,\eta\}$, where $[\cdot]$ denotes the projective class of the measured lamination. Before computing $\theta(t)$, and determining that $\xi=[\lambda^{-}]$, we describe the method in words. 

To compute the bending angle \(\theta(t)\), it is necessary to find a pair of distinct maximal support planes that intersect along the axis of the hyperbolic M\"{o}bius map \(\rho_{t}(\xi)\), i.e.~two half-planes that meet along the axis of $\rho_{t}(\xi)$. This can be achieved by considering the two other lifts of $\xi$ which are the axes of $\rho_{t}\left(\delta_{1}^{-1}\xi\delta_{1}\right)$ and $\rho_{t}\left(\delta_{4}\xi\delta_{4}^{-1}\right)$: One can check easily that \(\left[\delta_{1}^{-1}\cdot\xi\right]\) and \(\left[\delta_{4}\cdot\xi\right]\)---the homotopy classes of the concatenations---are simple closed curves. Thus, for the pair $\left\{\xi,\delta_{1}^{-1}\cdot\xi\right\}$ for example, there is a connected fundamental domain for the action of $\pi_{1}\Sigma$ that has the lifts of these curves on its boundary (motivating the term `neighbors' for this pair, see Figure \ref{neighbors}). The same is true, of course, for $\left\{\xi,\delta_{4}\cdot\xi\right\}$.

\begin{figure}[h]
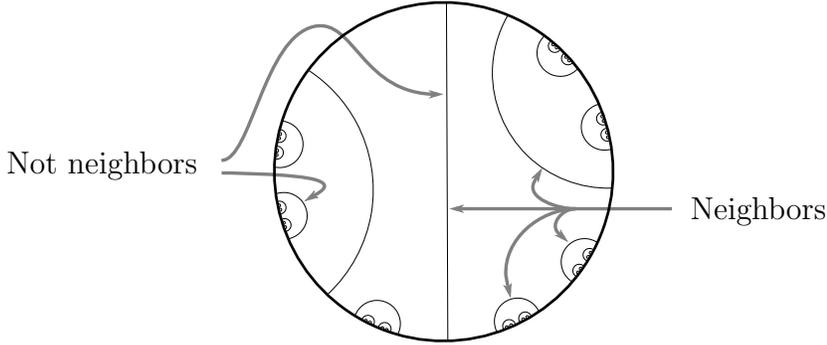

	\centering
	\begin{lpic}[clean]{neighboringliftssend(6cm)}
		\large
		\lbl[]{225,55;Neighbors}
		\lbl[]{-50,74;Not neighbors}
		
	\end{lpic}
	\caption{Lifts for $\xi\in\pi_{1}\Sigma$}
	\label{neighbors}
\end{figure}
The axes of each pair differ by a M\"{o}bius transformation, so their fixed points lie on a round circle. The round circle intersects $\Lambda_{\Gamma_{t}}$ in the four fixed points---non-elliptic elements have fixed points contained in the limit set---and since we know that $\xi$ is the support of a bending lamination for $\Gamma_{t}$, this circle is a maximal support plane for $\Gamma_{t}$. The two pairs of axes determine two maximal support planes intersecting along the axis $\rho_{t}(\xi)$, as desired. 

We thus consider three pairs of fixed points, corresponding to the three chosen lifts of $\xi$: $\mathrm{Fix}\;\rho_{t}\left(\delta_{1}^{-1}\xi\delta_{1}\right)$, $\mathrm{Fix}\;\rho_{t}(\xi)$, and $\mathrm{Fix}\;\rho_{t}\left(\delta_{4}\xi\delta_{4}^{-1}\right)$. The first four points determine one half-plane in $\H^{3}$, and the last four determine a second. We may now compute the bending angle between these half-planes using the argument of a cross-ratio. 

There is an ambiguity in our calculation of support planes: we do not know which side of the totally geodesic plane through a pair of neighboring axes is actually a supporting half-space for either of the groups $\Gamma_{t}$ or $\hat{\Gamma}_{t}$. In order, then, to distinguish between the top and the bottom laminations we will use the following consequence of the definitions: if there exists a support plane for a component $\Omega_{0}\subset\Omega_{\Gamma_{t}}$ which is not simultaneously a support plane for $\hat{\Gamma}_{t}$, then $\Omega_{0}$ is the bottom. Such a support plane is produced explicitly below at $t=\frac{1}{2}$, for the domain facing the convex core boundary component with bending supported on $\xi$.

In fact, for our purposes, it is enough to distinguish the top lamination from the bottom lamination at \(t=\frac{1}{2}\): The bending lamination map from \(\GF(M)\) to \(\ML(\Sigma)\) is continuous \cite[Thm.~4.6.]{keen-series}, and its image along the path $\{[\rho_{t}]\hspace{.1cm}|\hspace{.1cm}t\in U\}$, by the argument above, is confined to the subset of $\ML(\Sigma)$ supported on either $\xi$ or $\eta$. This subset is a pair of rays, intersecting only at the `zero' lamination. Thus, for instance, if $\xi$ switches at some point from the support for the bottom lamination of $\Gamma_{t}$ to the support for the top lamination of $\Gamma_{t}$, the Kleinian group $\hat{\Gamma}_{t}$ must have the `zero' bending lamination, i.e.~totally geodesic boundary. In this case, $\Gamma_{t}$ would be Fuchsian at this value of $t$. By Lemma \ref{param check}, \([\rho_{t}]\) is not Fuchsian for \(t<1\), hence the lamination on top at \(t=\frac{1}{2}\) must remain on top for all \(t<1\).

We proceed with the calculation of \(\theta(t)\). Denote the fixed points of $\rho_{t}(\xi)$ by $p_{t}^{\pm}$, where the choice is fixed by asking that the root in the expression of the fixed points be positive for $p^{+}$ and negative for $p^{-}$. The fixed points of $\rho_{t}\left(\delta_{1}^{-1}\xi\delta_{1}\right)$ and $\rho_{t}\left(\delta_{4}\xi\delta_{4}^{-1}\right)$ are thus, respectively, $\rho_{t}\left(\delta_{1}^{-1}\right)\cdot p_{t}^{\pm}$, and $\rho_{t}(\delta_{4})\cdot p_{t}^{\pm}$. The root below refers to the positive one.
\begin{figure}[h]
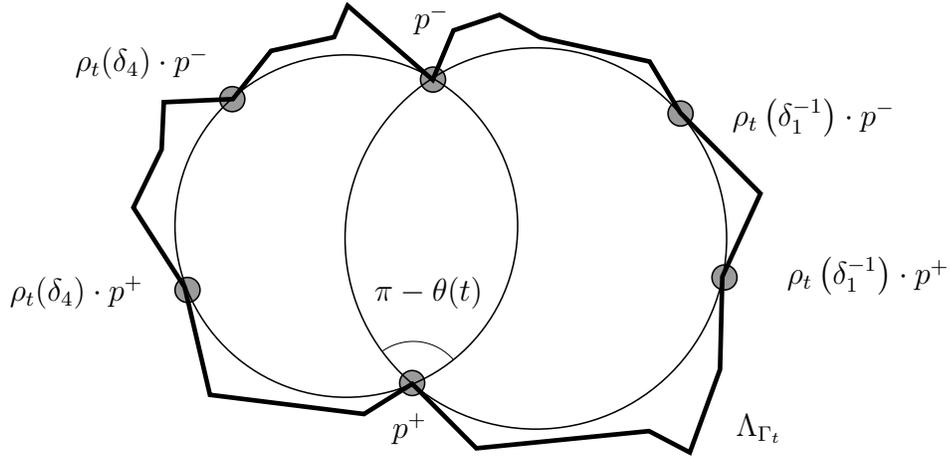

	\centering
		\begin{lpic}[clean]{angcalcsend(8.4cm)}
		\large
		\lbl[]{51,05;$p^{+}$}
		\lbl[]{55,80;$p^{-}$}
					\lbl[]{54.5,30;$\pi-\theta(t)$}
		\lbl[]{-10,30;$\rho_{t}(\delta_{4})\cdot p^{+}$}
		\lbl[]{2,72;$\rho_{t}(\delta_{4})\cdot p^{-}$}
		\lbl[]{135,33;$\rho_{t}\left(\delta_{1}^{-1}\right)\cdot p^{+}$}
		\lbl[]{125,62;$\rho_{t}\left(\delta_{1}^{-1}\right)\cdot p^{-}$}
		\lbl[]{115,5;$\Lambda_{\Gamma_{t}}$}
		\end{lpic}
	\caption{Maximal support planes for $\Gamma_{t}$, and the bending angle $\theta(t)$}
	\label{angcalc}
\end{figure}

We use the notation for the cross-ratio \(\;[a:b:c:d]=\frac{d-a}{d-c}\frac{b-c}{b-a}\). 
\begin{align*}
&&\left[\right. & \left.p_{t}^{+}:\rho_{t}(\delta_{4})\cdot p_{t}^{+}:p_{t}^{-}:\rho_{t}\left(\delta_{1}^{-1}\right)\cdot p_{t}^{+}\right]=\frac{1}{2t(-i+t)(i+t)^{3}}\cdot\\
&&& \left(-i-2t-2it^{2}-13i t^{4}+2t^{5}+4it^{6}+\right.\\
&&& \left.\left(-i-2t-it^{2}+2t^{3}\right)\sqrt{1+6t^{2}+17t^{4}-4t^{6}}\right)\\
&&&=\;\frac{1+5t^{2}+\sqrt{1+6t^{2}+17t^{4}-4t^{6}}}{2t(t^{2}+1)^{3}}\cdot\\
&&&\left(2t^{3}(t^{2}-3)+\;i\;(1-t^{2})\sqrt{1+6t^{2}+17t^{4}-4t^{6}}\right).
\end{align*}

It is easy to see that the imaginary part of this cross-ratio, for $t\in(0,1)$, is positive (cf. (\ref{expr:positive})). Thus \(\theta(t)=\pi-\arg\;\left[p_{t}^{+}:\rho_{t}(\delta_{4})\cdot p_{t}^{+}:p_{t}^{-}:\rho_{t}\left(\delta_{1}^{-1}\right)\cdot p_{t}^{+}\right]\). Taking the argument of the expression above gives
\begin{align*}
\theta(t) = \cos^{-1}\left(\frac{2t^{3}(3-t^{2})}{(1+t^{2})^{2}}\right).
\end{align*}

It remains to distinguish the top of the convex core boundary from the bottom. Following the outline described above, we fix some notation at the point \(t=\frac{1}{2}\). Let \(S\) be the circle through the points \(p_{\frac{1}{2}}^{+}\), \(p_{\frac{1}{2}}^{-}\), and \(\rho_{\frac{1}{2}}(\delta_{1}^{-1})\cdot p_{\frac{1}{2}}^{+}\), let its center be \(c\) and its radius \(r\). Noting that \(B_{\frac{1}{2}}\) is loxodromic, let its fixed points be \(p_{1},p_{2}\in\Lambda_{\hat{\Gamma}_{\frac{1}{2}}}$. The following computations are straightforward. The root is chosen with positive imaginary part.
\[c=-\frac{1}{2}-i\;;\hspace{.3cm}
r=\frac{\sqrt{5}}{2}\;;\hspace{.3cm}
p_{1}=\left(\frac{1+2i}{5}\right)\sqrt{7+i}\;;\hspace{.3cm}
p_{2}=-\left(\frac{1+2i}{5}\right)\sqrt{7+i}\]
\begin{align*}
|c-p_{1}|^{2}&=\frac{5}{4}+\sqrt{2}+\Re\sqrt{7+i}>\frac{5}{4}=r^{2}\\
|c-p_{2}|^{2}&=\frac{5}{4}+\sqrt{2}-\Re\sqrt{7+i}<\frac{5}{4}=r^{2}
\end{align*}
(Note that $\Re\sqrt{7+i}=\sqrt{\frac{7+5\sqrt{2}}{2}}>\sqrt{2}$.)

\begin{figure}[h]
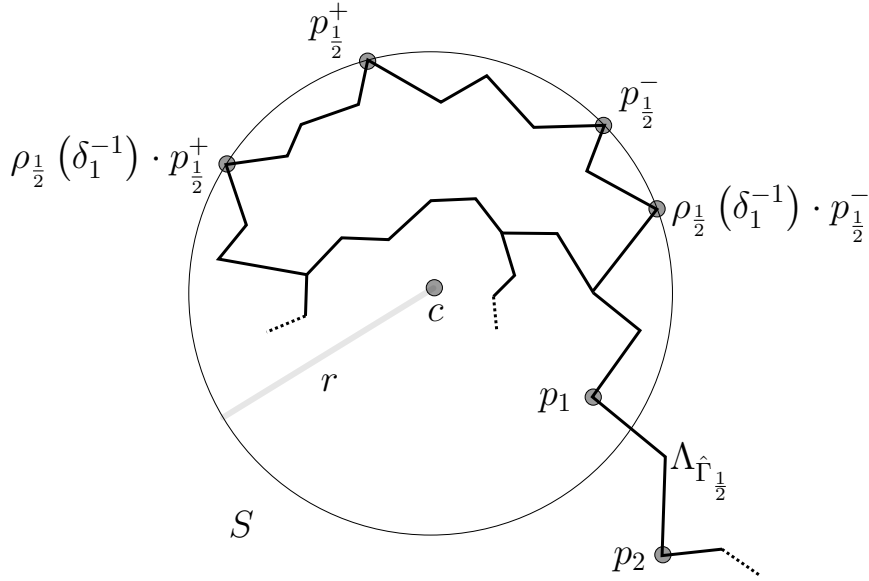

\vspace{.5cm}
	\centering
		\begin{lpic}[clean]{botlamsend(,7.cm)}
			\Large
			\lbl[]{70,75;$c$}
			\lbl[]{40,55;$r$}
			\lbl[]{15,15;$S$}
							\lbl[]{165,103;$\rho_{\frac{1}{2}}\left(\delta_{1}^{-1}\right)\cdot p_{\frac{1}{2}}^{-}$}
			\lbl[]{104,50;$p_{1}$}
			\lbl[]{125,5;$p_{2}$}
			\lbl[]{40,155;$p_{\frac{1}{2}}^{+}$}
			\lbl[]{128,135;$p_{\frac{1}{2}}^{-}$}
			\lbl[]{-22,117;$\rho_{\frac{1}{2}}\left(\delta_{1}^{-1}\right)\cdot p_{\frac{1}{2}}^{+}$}
							\lbl[]{145,30;$\Lambda_{\hat{\Gamma}_{\frac{1}{2}}}$}
		\end{lpic}
		\caption{A support plane for $\Gamma_{\frac{1}{2}}$ fails to be a support plane for $\hat{\Gamma}_{\frac{1}{2}}$}
		\label{botlam}
\end{figure}

A support plane for $\hat{\Gamma}_{\frac{1}{2}}$ must be the boundary of a supporting half-space, which must intersect $\CP^{1}$ in the domain of discontinuity. Because there are points $p_{1},p_{2}\in\Lambda_{\hat{\Gamma}_{\frac{1}{2}}}$ in the interior and exterior of the round disk $S$, the geodesic plane which meets $\CP^{1}$ in $S$ cannot be the boundary of a supporting half-space for $\hat{\Gamma}_{\frac{1}{2}}$. Because $S$ is a support plane for $\Gamma_{\frac{1}{2}}$ but not for $\hat{\Gamma}_{\frac{1}{2}}$, $S$ is a support plane for the bottom of $\Gamma_{\frac{1}{2}}$. (See Figure \ref{botlam} for a schematic where the exterior of $S$ is a supporting half-space for the bottom of $\Gamma_{\frac{1}{2}}$). Since $S$ is on the side of $\Gamma_{\frac{1}{2}}$ with bending lamination $\xi$, this implies that \(\xi\) is the support of the bending lamination on the bottom.}
\end{proof}

For $\gamma\in\pi_{1}(M)$, let $\ell(\gamma,\Gamma_{t})$ denote the hyperbolic translation length of $\rho_{t}(\gamma)$ in $\H^{3}$. Using Lemma \ref{bend cc} we may now find $\inf U$ explicitly.

\begin{lemma}{The bottom (resp. top) convex core boundary surface of $\Gamma_{t}$ is the element of $\Rh_{\{\xi,\eta\}}$ determined by $\ell(\xi,\Gamma_{t})$ (resp. $\ell(\eta,\Gamma_{t})$), where

\begin{equation}
\label{ell}
\ell(\xi,\Gamma_{t})=2\;\cosh^{-1}\left(\frac{1+8t^{2}+21t^{4}-2t^{6}}{2t^{2}(1+t^{2})^{2}}\right)
\end{equation}
and
$$\ell(\eta,\Gamma_{t})=2\;
\cosh^{-1}\left(\frac{-1+4t^{2}+74t^{4}+196t^{6}-t^{8}}{(1+t^{2})^{4}}\right).$$
Moreover, $\frac{1}{2}\left(5+3\sqrt{3}-\sqrt{44+26\sqrt{3}}\right)=\inf U$.}
\label{qf info}
\end{lemma}
\begin{proof}
{For a loxodromic isometry $A\in\PSL_{2}\C$, its translation length is given by $\ell(A)=2\cosh^{-1}\left(\frac{|trA|}{2}\right)$. Thus equations (\ref{ell}) can be computed directly. For $t\in U$, the hyperbolic length of the curve which supports the bending lamination of $\Gamma_{t}$, on the hyperbolic surface on the top or the bottom of the convex core boundary, is precisely the length of the curve in the quasi-Fuchsian manifold $\nicefrac{\H^{3}}{\Gamma_{t}}$. By Lemma \ref{param rhomb}, this quantity determines the convex core boundary surface. 

Let $t_{0}=\inf\;U$. Because $\Sigma$ is incompressible in $M$, by \cite[Thm.~A]{bonahon} the hyperbolic manifold $\nicefrac{\H^{3}}{\hat{\Gamma}_{t_{0}}}$ is {\it geometrically tame}. In particular, it has an end invariant $\lambda$, a geodesic lamination which is the union of limits of simple closed geodesics exiting the end and/or simple closed curves that are parabolic in $\hat{\Gamma}_{t_{0}}$. Since $\hat{\Gamma}_{t_{0}}\in\mathrm{Fix}\;\Psi$, the invariant $\lambda$ must also be preserved by the rhombic symmetry. Because the natural map from $\PML(\Sigma)$ to geodesic laminations is equivariant for the action of $MCG^{*}(\Sigma)$, Lemma \ref{fix lam} implies that $\lambda\in\{\xi,\eta\}$.

Thus $\hat{\Gamma}_{t_{0}}$ is geometrically finite with end invariant either $\xi$ or $\eta$. 

We check below the following computational facts:
\begin{enumerate}
\item $|tr\tilde{\rho}_{t}(\xi)|>2$ for all $t\in\left(0,1\right]$. 
\item For $1>t>\frac{1}{2}\left(5+3\sqrt{3}-\sqrt{44+26\sqrt{3}}\right)$, we have $|tr\tilde{\rho}_{t}(\eta)|>2$, and $tr^{2}\rho_{t}(\eta)=4$ at $t=\frac{1}{2}\left(5+3\sqrt{3}-\sqrt{44+26\sqrt{3}}\right)$. 
\end{enumerate}

As a result of (1), the end invariant of $\hat{\Gamma}_{t_{0}}$ must be $\eta$. As a result of (2), we may compute explicitly the lower bound $t_{0}=\frac{1}{2}\left(5+3\sqrt{3}-\sqrt{44+26\sqrt{3}}\right)$.
\vspace{.2cm}\\
\noindent(1): For all $t\in(0,1]$,
\begin{align}\label{expr:positive}
2+tr\tilde{\rho}_{t}(\xi)&=\frac{-1-8t^{2}-21t^{4}+2t^{6}}{t^{2}(1+t^{2})^{2}}+2\\\nonumber&=\frac{-1-6t^{2}-17t^{4}+4t^{6}}{t^{2}(1+t^{2})^{2}}\\\nonumber
&=\frac{-1-6t^{2}-13t^{4}-4t^{4}(1-t^{2})}{t^{2}(1+t^{2})^{2}}<0.
\end{align}
\vspace{.2cm}\\
\noindent(2): For all $t\in(t_{0},1]$,
\begin{align*}
2+tr\tilde{\rho}_{t}(\eta)&=\frac{2(1-4t^{2}-74t^{4}-196t^{6}+t^{8})}{(1+t^{2})^{4}}+2\\&=\frac{4(1-34t^{4}-96t^{6}+t^{8})}{(1+t^{2})^{4}}\\
&=\frac{4\left((1+t^{2})^{2}-2t(1+5t^{2})\right)\left((1+t^{2})^{2}+2t(1+5t^{2})\right)}{(1+t^{2})^{4}}.
\end{align*}
One may check directly (using the explicit equation for the zeros of a quartic) that $(1+t^{2})^{2}-2t(1+5t^{2})$ has two real zeros, the smaller of which is $t_{0}=\frac{1}{2}\left(5+3\sqrt{3}-\sqrt{44+26\sqrt{3}}\right)$. The larger zero, $\frac{1}{2}\left(5+3\sqrt{3}+\sqrt{44+26\sqrt{3}}\right)$, is greater than 1, and for $t=1$, $(1+t^{2})^{2}-2t(1+5t^{2})$ is negative. Thus $2+tr\tilde{\rho}_{t}(\eta)<0$ for all $t\in(t_{0},1]$, and $tr^{2}\rho_{t_{0}}(\eta)=4$.}
\end{proof}  

For $t\in(t_{0},1]$, recall the notation $[\Gamma_{t}]=Q(X_{t},Z_{t})$, so that in particular $\sigma_{M}(X_{t})=Z_{t}$. Denote the bottom surface of the convex core boundary of $\Gamma_{t}$ by $Y_{t}$. That is, recalling Lemma \ref{param rhomb}, $$Y_{t}=\left(\ell_{\xi}|_{\Rh_{\{\xi,\eta\}}}\right)^{-1}\left(\ell(\xi,\Gamma_{t})\right).$$

\section{Non-Injectivity of $t\mapsto Z_{t}$}
\label{non-injectivity}

In this section we use a description of $Z_{t}$ as a grafted surface to show the non-monotonicity of the function $Ext_{\xi}Z_{\--}:(t_{0},1]\to\R^{+}$. Since $Z_{t}\in\Rh_{\{\xi,\eta\}}$ by Lemma \ref{sym qf}, and since $Ext_{\xi}:\Rh_{\{\xi,\eta\}}\to\R^{+}$ is a diffeomorphism by Lemma \ref{param rhomb}, the non-monotoniciy of $Ext_{\xi}Z_{\--}$ implies the non-injectivity of $t\mapsto Z_{t}$, for $t\in(t_{0},1]$. 

Recall that {\it grafting} is a map (see \cite{kamishima-tan} and \cite{mcmullen2} for details)
$$gr:\ML(\Sigma)\times\T(\Sigma)\to\T(\Sigma).$$ Briefly, for $\gamma\in\mathcal{S}$, the Riemann surface $gr(\tau\cdot\gamma,X)$ is obtained by cutting open $X$ along the geodesic representative for $\gamma$, and inserting a Euclidean annulus of height $\tau$.

For a geometrically finite hyperbolic 3-manifold $N$, with conformal boundary $X_{\infty}\in\T(\partial N)$, convex core boundary $X_{cc}\in\T(\partial N)$, and bending lamination $\lambda\in\ML(\partial N)$, grafting provides the description:
$$X_{\infty}=gr(\lambda,X_{cc}).$$
Employing Lemmas \ref{bend cc} and \ref{qf info}, we thus have
\begin{equation}
\label{image eqn}
Z_{t}=gr(\theta(t)\cdot\xi,Y_{t}).
\end{equation}
We note that there is also a projective version of grafting, $Gr$, that `covers' conformal grafting ~\cite{kamishima-tan}. This provides the natural (quasi-Fuchsian) projective structure $\mathcal{Z}_{t}=Gr(\theta(t)\cdot\xi,Y_{t})$ which the surface $Z_{t}$ inherits as the boundary of $\nicefrac{\H^{3}}{\Gamma_{t}}$.

In fact, it will be more direct to deal with a conformal invariant closely related to extremal length, the {\it modulus} of a quadrilateral (see \cite{lehto}, \cite{ahlfors} for details). For $X\in\Rh_{\{\xi,\eta\}}$, consider the rhombic coordinate on $X$ (see Figure \ref{rhombpic1}). Let $Q_{X}$ denote the quadrilateral given by $\{z\in\CP^{1}\;|\;\text{Re}(z),\text{Im}(z)>0\}$, with vertices $\{is,0,1,\infty\}$, and horizontal sides given by the arcs $(0,1)$ in the positive real axis and $(is,\infty)$ in the positive imaginary axis. Denote the modulus of a quadrilateral, or annulus, by $Mod(\cdot)$.

\begin{lemma}{We have $Ext_{\xi}X=4\;Mod(Q_{X})$.}
\label{ext rhomb}
\end{lemma}

\begin{proof}{
The quantity $Ext_{\xi}X$ can also be computed as the modulus of the unique maximal modulus annulus containing $\xi$ as its core curve. Consider this annulus $A_{\xi}$. Since $\xi$ is preserved by the anti-holomorphic maps $z\mapsto\overline{z}$ and $z\mapsto-\overline{z}$, the annulus $A_{\xi}$ is preserved by these maps. This is enough to ensure that $A_{\xi}=\C\setminus\left((-\infty,-1)\cup(1,\infty)\cup(-is,is)\right)$. (Alternatively, in \cite[p.~23, II.]{ahlfors}, there is a classical extremal length problem whose solution---due to Teichm\"{u}ller---implies that $A_{\xi}$ has the given form).

Since the curves 
\begin{itemize}
\item $\{i(s+\tau)\;|\;\tau\in(0,\infty)\}$
\item $\{-i(s+\tau)\;|\;\tau\in(0,\infty)\}$
\item $\{\tau\;|\;\tau\in(0,1)\}$
\item $\{-\tau\;|\;\tau\in(0,1)\}$
\end{itemize}
are all fixed by $z\mapsto\overline{z}$ and $z\mapsto-\overline{z}$, they must be vertical geodesics in the Euclidean metric for $A_{\xi}$. By \cite[p.~16]{ahlfors}, the modulus $Mod(A_{\xi})$ is equal to the sum of the moduli of the four quadrilateral pieces obtained after cutting along these vertical geodesics. Since these four quadrilateral pieces are each congruent to $Q_{X}$, we have $Mod(A_{\xi})=4Mod(Q_{X})$.}\end{proof}

For brevity, let $Q_{t}:=Q_{Z_{t}}$ (see Figure \ref{xiQ}).

\begin{figure}[h]
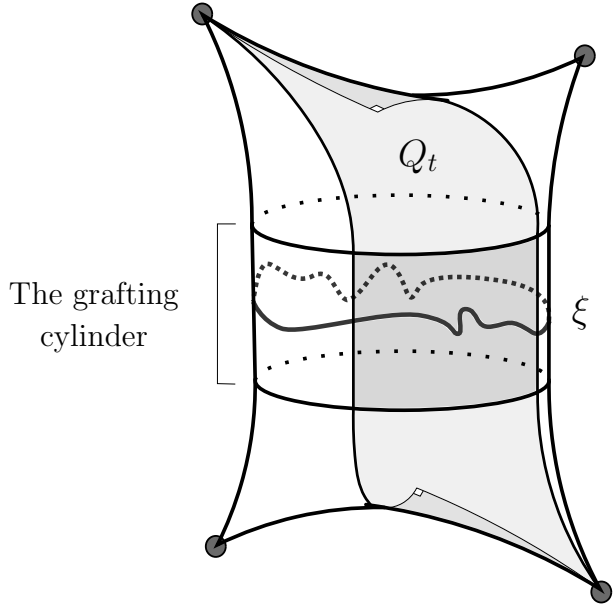

	\centering
	\begin{lpic}[clean]{xiQsend(,8cm)}
	\Large
	\lbl[]{120,90;$\xi$}
		\lbl[]{70,138;$Q_{t}$}
	\large
	\lbl[]{-30,95;The grafting}
	\lbl[]{-30,82;cylinder}
	\end{lpic}
	\caption{The quadrilateral $Q_{t}\subset Z_{t}$ and the curve $\xi$}
	\label{xiQ}
\end{figure}

In Lemma \ref{proj coord Q}, we present the natural projective structure on $Q_{t}$, given by the projective grafting description of the projective structure $\mathcal{Z}_{t}$ on $Z_{t}$. Namely, we choose a coordinate so that $\xi$ has holonomy $z\mapsto e^{\ell(\xi,\Gamma_{t})}z$. Let $D(z,r)$ be the open disk in $\C$ centered at $z$ of radius $r$, let $E^{c}$ indicate the complement of a set $E\subset\C$, and let $L(t):=\frac{1}{4}\ell(\xi,\Gamma_{t})$. The proof of Lemma \ref{proj coord Q} is a straightforward interpretation of projective grafting, so we omit it.

\begin{lemma}{The projective structure on the quadrilateral $Q_{t}$ is given by the region (see Figure \ref{Qt})
$$Q_{t}\cong\overline{D(c_{1},r_{1})}\cap\left(\bigcap\limits_{i=2}\limits^{4} D(c_{i},r_{i})^{c}\right)$$
with `vertical' sides given by the two concentric arcs, such that:
\begin{enumerate}
\item{$(c_{1},r_{1})=\left(0,e^{L(t)}\right)$}
\item{$(c_{2},r_{2})=(0,1)$}
\item{$(c_{3},r_{3})=\left(e^{L(t)}\cosh(L(t)),
e^{L(t)}\sinh(L(t))\right)$}
\item{$(c_{4},r_{4})=\left(e^{i\left(\pi+\theta(t)\right)}\cosh(L(t)),
\sinh(L(t))\right)$}
\end{enumerate}}
\label{proj coord Q}
\end{lemma}

\begin{figure}[h]
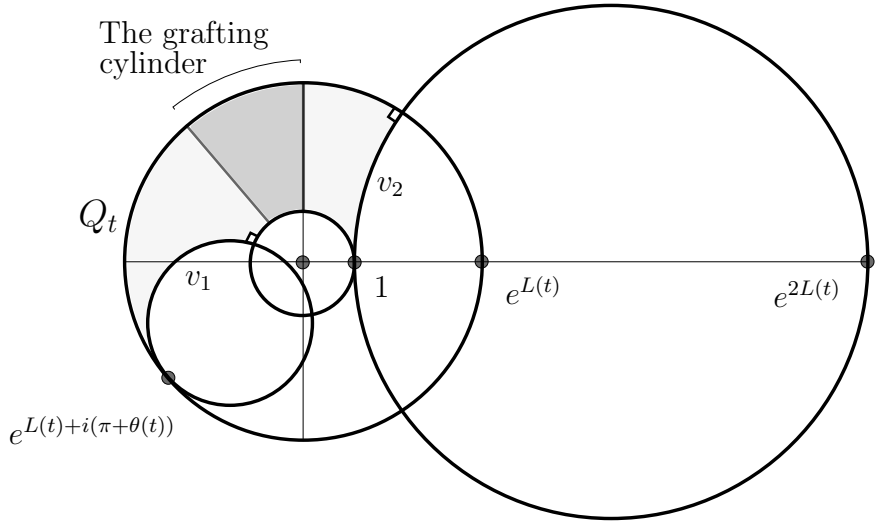

	\centering
	\begin{lpic}[clean]{Dcrsend(10cm)}
		\large
					
					\lbl[]{85,77;$1$}
					\lbl[]{135,76;$e^{L(t)}$}
					\lbl[]{225,75;$e^{2L(t)}$}
					\lbl[]{-10,30;$e^{L(t)+i(\pi+\theta(t))}$}
									\lbl[]{20,160;The grafting}
									\lbl[]{10,150;cylinder}
											\lbl[]{25,79;$v_{1}$}
											\lbl[]{88,110;$v_{2}$}
				\Large
					\lbl[]{-8,100;$Q_{t}$}
					
	\end{lpic}
	\caption{A picture of the quasi-Fuchsian structure on $Q_{t}$. The vertical sides are labelled $v_{1}$ and $v_{2}$.}
	\label{Qt}
\end{figure}

We now have a function, $t\mapsto Mod(Q_{t})$, for $t\in(t_{0},1]$, whose non-monotonicity would suffice to show the non-injectivity of $t\mapsto Z_{t}$. As a result of Lemma \ref{proj coord Q}, it is possible to produce numerical estimates of the modulus that confirm non-monotonicity: One may use conformal maps to `open' the two punctures to right angles, and a carefully scaled logarithm to make the quadrilateral look `nearly rectangular'. These produce the rough estimates:
\begin{itemize}
\item $Mod\left(Q_{1}\right)\approx2.3$
\item $Mod\left(Q_{.52}\right)\approx1.8$
\item $Mod\left(Q_{.39}\right)\approx2.0$
\end{itemize}
(Recall that $t_{0}=\frac{1}{2}\left(5+3\sqrt{3}-\sqrt{44+26\sqrt{3}}\right)$, and note that $.39>t_{0}$). Unfortunately, controlling the error in these approximations quickly becomes very delicate, so we pursue a different approach. In Lemma \ref{coord Q F}, we apply a normalizing transformation to the conformal structure given in Lemma \ref{proj coord Q} which allows a direct comparison of moduli of quadrilaterals. 

While there remain some involved calculations, this method reduces the computational difficulties considerably. After Lemma \ref{coord Q F}, we have a reasonable list of verifications to check, involving integers and simple functions. We defer this list of verifications to \S\ref{compute} and to the Appendix, as they contribute no new ideas to the proof. 

Choose a branch of the logarithm $\arg(z)\in(0,2\pi)$. Applying $z\mapsto\log z$ takes $Q_{t}$ into a vertical strip, with its vertical sides sent into a pair of vertical lines. Its horizontal sides can be described as the graphs of two explicit functions over the interval between the vertical sides. Following this map by $z\mapsto\frac{1}{L(t)}(z-\frac{\pi+\theta(t)}{2}i)$ we arrive at $R_{t}$, a certain conformal presentation of $Q_{t}$ (see Figure \ref{Qtnorm}). The proof of the following lemma is a straightforward calculation. 

\begin{figure}[h]
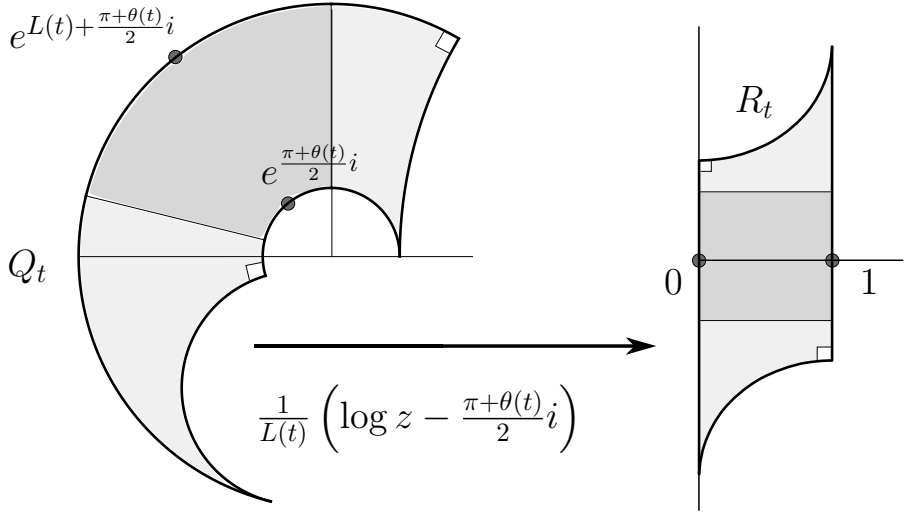

	\centering
	\begin{lpic}[clean]{Qtmapsend(11cm)}
		\Large
		\lbl[]{110,30;$\frac{1}{L(t)}\left(\log z-\frac{\pi+\theta(t)}{2}i\right)$}
		\lbl[]{192,74;$0$}
				\lbl[]{218,132;$R_{t}$}
		\lbl[]{255,74;$1$}
		\lbl[]{5,156;$e^{L(t)+\frac{\pi+\theta(t)}{2}i}$}
		\lbl[]{75,112;$e^{\frac{\pi+\theta(t)}{2}i}$}
		\lbl[]{-16,79;$Q_{t}$}
	\end{lpic}
	\caption{The normalizing transformation $Q_{t}\to R_{t}$}
	\label{Qtnorm}
\end{figure}

\begin{lemma}{The image under the transformation $z\mapsto\frac{1}{L(t)}\left(\log z-\frac{\pi+\theta(t)}{2}i\right)$ of $Q_{t}$ is given by $R_{t}$ (see Figure \ref{Qtnorm}). The quadrilateral $R_{t}$ has vertical sides in the vertical lines $\{\Re z=0\}$, $\{\Re z=1\}$, and horizontal sides given as the graphs of the functions $x\mapsto F(x,t)$ and $x\mapsto-F(1-x,t)$ over $x\in[0,1]$. The function $F(x,t)$ is given by:
\begin{equation}
\label{F}
F(x,t)=\alpha(t)-\beta(x,t)
\end{equation}
where $\alpha$ and $\beta$ are given by:
\begin{equation}
\label{alpha}
\alpha(t)= \frac{\pi+\theta(t)}{2L(t)}
\end{equation}
\begin{equation}
\label{beta}
\beta(x,t)=\frac{1}{L(t)}\;\cos^{-1}\left(\frac{\cosh\;(xL(t))}{\cosh\;(L(t))}\right)
\end{equation}}
\label{coord Q F}
\end{lemma}

The essentially useful fact in Lemma \ref{coord Q F} is that $Mod(Q_{t})$ now depends only on the graph of the function $x\mapsto F(x,t)$ over $x\in[0,1]$, which forms a kind of `profile' for the quadrilateral $Q_{t}$. There is some geometric intuition to this profile (cf. ~\cite[p.~21]{mcmullen2}): In the cover $\widetilde{Z_{t}}$ of $Z_{t}$, corresponding to the subgroup $\langle\xi\rangle\;\textless\;\pi_{1}(\Sigma)$, one has 
$$Ext_{\xi}\widetilde{Z_{t}}=\frac{4L(t)}{\pi+\theta(t)}=\frac{2}{\alpha(t)}.$$
Since the maximal modulus annulus with core curve homotopic to $\xi$ can be lifted to $\widetilde{Z_{t}}$, one has, via Rengel's Inequality, 
$$Ext_{\xi}Z_{t}\le Ext_{\xi}\widetilde{Z_{t}}.$$
In other words,
\begin{equation}
\label{mod ineq}
Mod(Q_{t})\;\le\;\frac{1}{2\alpha(t)}.
\end{equation}

Note that $F(x,t)$ depends on $x$ only through $\beta$, as $\alpha$ is independent of $x$. In fact, $\frac{1}{2\alpha(t)}$ is the modulus of the rectangle sharing its vertical sides with $Q_{t}$ and its horizontal sides contained in the lines $\{\Im z=\alpha(t)\}$ and $\{\Im z=-\alpha(t)\}$. Thus $\beta$ accounts for the discrepancy in inequality (\ref{mod ineq}).

Some computational details, deferred to \S\ref{compute}, provide the following:

\begin{lemma}{We have the containment of quadrilaterals $R_{1}\subset R_{\frac{1}{2}}$, with the vertical sides of $R_{1}$ contained in the vertical sides of $R_{\frac{1}{2}}$.}
\label{R 1}
\end{lemma}

\begin{lemma}{We have the containment of quadrilaterals $R_{\frac{2}{5}}\subset R_{\frac{1}{2}}$, with the vertical sides of $R_{\frac{2}{5}}$ contained in the vertical sides of $R_{\frac{1}{2}}$.}
\label{R 2/5}
\end{lemma}

These Lemmas directly imply non-monotonicity of $t\mapsto Mod(Q_{t})$.

\begin{prop}{The function $t\mapsto Mod\left(Q_{t}\right)$ is not monotone on the interval $(t_{0},1]$.}
\label{mod non-mono}
\end{prop}

\begin{proof}{Note that $t_{0}<\frac{2}{5}$, by Lemma \ref{t0}. By Lemmas \ref{R 1} and \ref{R 2/5} (see Figure \ref{Qtnonmono}), we have:
\begin{align*}
Mod(Q_{1}) & \le Mod\left(Q_{\frac{1}{2}}\right)\\
Mod\left(Q_{\frac{2}{5}}\right) & \le Mod\left(Q_{\frac{1}{2}}\right)
\end{align*}}\end{proof}
 
\begin{figure}[h]
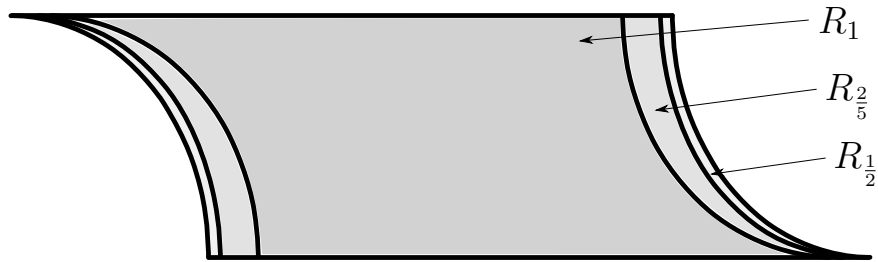

	\centering
	\begin{lpic}[clean]{Qtnonmonosend(11.5cm)}
				\Large
				\lbl[]{150,17;$R_{\frac{1}{2}}$}
				\lbl[]{148.5,29;$R_{\frac{2}{5}}$}
				\lbl[]{147,42;$R_{1}$}
	\end{lpic}
	\caption{The non-monotonicity of $Mod\left(Q_{t}\right)$. The quadrilateral $R_{t}$ is pictured rotated so that its vertical sides appear horizontal.}
	\label{Qtnonmono}
\end{figure}

The non-monotonicity of Proposition \ref{mod non-mono} implies the existence of a critical point for the skinning map that is the subject of Theorem \ref{thm}, so we are now ready to give the proof.

\begin{proof}[{Proof of Theorem \ref{thm}}]
{The anti-holomorphic map $\sigma_{M}:\H\to\H$, by equation (\ref{image eqn}), sends the real 1-dimensional submanifold $\Rh:=\Rh_{\{\xi,\eta\}}$ to itself. By Lemmas \ref{param rhomb} and \ref{mod non-mono}, $\left(Ext_{\xi}\circ\sigma_{M}\right)|_{\Rh}$ is a non-monotonic continuously differentiable function, and thus $\left(Ext_{\xi}\circ\sigma_{M}\right)|_{\Rh}$ has a critical point. By Lemma \ref{param rhomb}, $Ext_{\xi}|_{\Rh}$ is a diffeomorphism, and thus $\sigma_{M}$ has a critical point.}
\end{proof}

\section{Computational Lemmas}
\label{compute}

The goal of this section is to prove Lemmas \ref{R 1} and \ref{R 2/5}, which together imply non-monotonicity of $Mod(Q_{t})$ (see Proposition \ref{mod non-mono}). Since the vertical sides of $R_{t}$ are contained in the lines $\{\Re z=0\}$ and $\{\Re z=1\}$, and the horizontal sides are graphs of functions over $x\in[0,1]$, the containments in these Lemmas are a consequence of inequalities involving the functions $x\mapsto F(x,t)$. In particular, we seek uniform estimates such as $F(x,t_{1})<F(x,t_{2})$, for $t_{1},t_{2}\in(t_{0},1]$, and for all $x\in[0,1]$. 

Essentially, it is the non-monotonicity of $\alpha(t)$ that accounts for such uniform estimates. However, while $\alpha$ is independent of $x$, $F$ is not, and in fact we may have $\alpha(t_{1})>\alpha(t_{2})$ while $F(x_{0},t_{1})<F(x_{0},t_{2})$, for some $x_{0}\in[0,1]$. We must therefore take more care in estimates on $F(x,t_{1})-F(x,t_{2})$ by accounting for the effect of $\beta$. 

While $\beta$ depends on $x$, we can control
\[\displaystyle\max_{x\in[0,1]}|\beta(x,t_{1})-\beta(x,t_{2})|.\]
This allows comparisons between $F(x,t_{1})$ and $F(x,t_{2})$ over all $x$. Some computational pieces that are easily checked with a computer are collected in the Appendix, where a Mathematica notebook containing a demonstration of these verifications is also indicated.

For ease in presentation, we introduce some notation for functions that will be used throughout this section:
\vspace{.1cm}
\begin{align}
\label{p1}
p_{1}(t)&=1+8t^{2}+21t^{4}-2t^{6}\\
\label{p2}
p_{2}(t)&=2t^{2}(1+t^{2})^{2}\\
\label{p3}
p_{3}(t)&=4t^{5}(3-t^{2})
\end{align}
In terms of these polynomials, equations (\ref{theta}--\ref{ell}) become:
\begin{align} 
\label{angps}
\theta(t)&=\cos^{-1}\left(\frac{p_{3}(t)}{p_{2}(t)}\right)\\
\label{Lps}
L(t)&=\frac{1}{2}\cosh^{-1}\left(\frac{p_{1}(t)}{p_{2}(t)}\right)
\end{align}
And recall (\ref{alpha}--\ref{beta}):
\begin{align}
\label{a}
\alpha(t)&=\frac{\pi+\theta(t)}{2L(t)}\\
\label{b}
\beta(x,t)&=\frac{1}{L(t)}
\cos^{-1}\left(\frac{\cosh\left(x\;L(t)\right)}{\cosh\left(L(t)\right)}\right)
\end{align}

The interested reader can refer to equations (\ref{p1}--\ref{b}) to check the computations below. The computational strategy is to first turn inequalities involving the $\cos^{-1}$ and $\cosh^{-1}$ functions into inequalities involving logarithms and constants consisting of algebraic numbers and $\pi$. These statements are then reduced to explicit inequalities involving large powers of algebraic numbers, which can be checked with a computer. Indeed, for the patient reader, they could all be checked by hand.

\begin{lemma}{We have $\alpha(\frac{1}{2})>\alpha(1)$.}
\label{alpha 1}
\end{lemma}

\begin{proof}{Calculating $\alpha(1)$ and $\alpha\left(\frac{1}{2}\right)$, and estimating $\cos^{-1}\left(\frac{11}{25}\right)$ using Lemma \ref{A1}, we have:
\begin{align*}
\alpha(1)&=\frac{\pi}{\cosh^{-1}\left(\frac{7}{2}\right)}\\
\alpha\left(\frac{1}{2}\right)&=\frac{\pi+\cos^{-1}\left(\frac{11}{25}\right)}{\cosh^{-1}\left(\frac{137}{25}\right)}>
\frac{\frac{27\pi}{20}}{\cosh^{-1}\left(\frac{137}{25}\right)}
\end{align*}
Now referring to Lemma \ref{A5},
\begin{align*}
\left(\frac{7+3\sqrt{5}}{2}\right)^{27}&>\left(\frac{137+36\sqrt{14}}{25}\right)^{20}.
\end{align*}
Taking logarithms and re-arranging, we find
\begin{align*}
\frac{\frac{27}{20}}{\log\left(\frac{137+36\sqrt{14}}{25}\right)} &> \frac{1}{\log\left(\frac{7+3\sqrt{5}}{2}\right)}.
\end{align*}
Since, for $x\ge1$, $\cosh^{-1}x=\log\left(x+\sqrt{x^{2}-1}\right)$, we have
\[\frac{\frac{27}{20}}{\cosh^{-1}\left(\frac{137}{25}\right)}>\frac{1}{\cosh^{-1}\left(\frac{7}{2}\right)}.\]
This implies $\alpha\left(\frac{1}{2}\right)>\alpha(1)$, as desired.
}
\end{proof}
The proof of Lemma \ref{alpha beta 2/5} proceeds along similar lines, so we suppress commentary.

\begin{lemma}{We have $\alpha(\frac{1}{2})-\alpha(\frac{2}{5})>\beta(0,\frac{1}{2})-\beta(0,\frac{2}{5})$}
\label{alpha beta 2/5}
\end{lemma}

\begin{proof}{\[
\alpha\left(\frac{1}{2}\right)-\alpha\left(\frac{2}{5}\right)>
\frac{\frac{27\pi}{20}}{\cosh^{-1}\left(\frac{137}{25}\right)}-
\frac{\frac{17\pi}{12}}{\cosh^{-1}\left(\frac{43897}{6728}\right)}
\hspace{.3cm}\text{ by Lemmas \ref{A1} and \ref{A2}}
\]
\vspace{.2cm}\\
\begin{align*}
\beta\left(0,\frac{1}{2}\right)-\beta\left(0,\frac{2}{5}\right)&=
\frac{2\cos^{-1}\left(\frac{5}{9}\right)}{\cosh^{-1}\left(\frac{137}{25}\right)}-
\frac{2\cos^{-1}\left(\frac{2}{3}\right)}{\cosh^{-1}\left(\frac{7}{2}\right)}
\\&<
\frac{\frac{19\pi}{30}}{\cosh^{-1}\left(\frac{137}{25}\right)}-
\frac{\frac{13\pi}{20}}{\cosh^{-1}\left(\frac{43897}{6728}\right)}
\hspace{.3cm}\text{ by Lemmas \ref{A3} and \ref{A4}}
\end{align*}
\vspace{.1cm}\\
\[
\left(\frac{137+36\sqrt{14}}{25}\right)^{46}<\left(\frac{43897+225\sqrt{37169}}{6728}\right)^{43}
\hspace{.4cm}\text{ by Lemma \ref{A7}}
\]
\begin{align*}
&\Longrightarrow
\frac{46}{43}=\frac{\frac{13}{20}-\frac{17}{12}}{\frac{19}{30}-\frac{27}{20}}<
\frac{\log\left(\frac{43897+225\sqrt{37169}}{6728}\right)}{\log\left(\frac{137+36\sqrt{14}}{25}\right)}=
\frac{\cosh^{-1}\left(\frac{43897}{6728}\right)}{\cosh^{-1}\left(\frac{137}{25}\right)}\\
&\Longrightarrow
\frac{\frac{27\pi}{20}}{\cosh^{-1}\left(\frac{137}{25}\right)}-
\frac{\frac{17\pi}{12}}{\cosh^{-1}\left(\frac{43897}{6728}\right)}>
\frac{\frac{19\pi}{30}}{\cosh^{-1}\left(\frac{137}{25}\right)}-
\frac{\frac{13\pi}{20}}{\cosh^{-1}\left(\frac{43897}{6728}\right)}
\end{align*}\\
The inequality follows.}
\end{proof}
\vspace{.1cm}

\begin{lemma}{$L$ is monotone decreasing on $(0,1]$.}
\label{L 1to1}
\end{lemma}

\begin{proof}{
Recall the definititions of $p_{1}(t)$ and $p_{2}(t)$, equations (\ref{p1}) and (\ref{p2}). 
We have:
\begin{align*}
p_{1}'(t)&=4t(4+21t^{2}-3t^{4})\\
p_{2}'(t)&=4t(1+4t^{2}+3t^{4})
\end{align*}

Since $t^{2}>t^{4}$ for $t\in(0,1)$, it is straightforward to check:

\begin{equation}
\label{p1p2}
p_{1}(t)>p_{2}(t)>0
\end{equation}
\begin{equation}
\label{p1'p2'}
p_{1}'(t)>p_{2}'(t)>0
\end{equation}
\vspace{.1cm}\\
Since $\cosh^{-1}$ is monotone increasing on $[1,\infty)$ it is thus enough to check that $\frac{p_{1}}{p_{2}}$ is monotone decreasing on $(0,1]$. This follows from equations (\ref{p1p2}) and (\ref{p1'p2'}).}
\end{proof}
\vspace{.1cm}

\begin{lemma}{For $t_{1},t_{2}\in(0,1)$ such that $t_{2}>t_{1}$, and for all $x\in(0,1)$, we have \[0<\beta(x,t_{2})-\beta(x,t_{1})<\beta(0,t_{2})-\beta(0,t_{1}).\]}
\label{beta all x}
\end{lemma}
\begin{proof}{We compute:

\begin{align*}
\frac{\partial^{2}\beta}{\partial x\partial t}(x,t)&=
\left[\frac
{-L'(t)\sinh(2L(t))\cosh(xL(t))}
{2\left(\cosh^{2}L(t)-\cosh^{2}(xL(t))\right)^{\frac{3}{2}}}\right]\cdot
\left[x\tanh L(t)-\tanh(xL(t))\right]
\end{align*}
By Lemma \ref{L 1to1}, $L'(t)<0$ for all $t\in(0,1)$ and the first bracketed term is positive.

The function $x^{-1}\tanh x$ is decreasing for $x>0$, so for $x\in(0,1)$ we have
\[\frac{\tanh(xL(t))}{xL(t)}>\frac{\tanh L(t)}{L(t)}\]
and the second bracketed term is negative. We conclude that $\frac{\partial^{2}\beta}{\partial x\partial t}(x,t)<0$ for all $t,x\in(0,1)$. 
\vspace{.1cm}

Thus on the domain $(0,1)\times(0,1)$, we have:
\begin{enumerate}
\renewcommand{\labelenumi}{(\roman{enumi})}
\item For fixed $x$, the function $\frac{\partial\beta}{\partial x}(x,t)$ is decreasing in $t$.
\vspace{.1cm}
\item For fixed $t$, the function $\frac{\partial\beta}{\partial t}(x,t)$ is decreasing in $x$.\end{enumerate}

A straightforward computation shows that
\[\displaystyle\lim_{x\to1}\frac{\partial\beta}{\partial t}(x,t)=0.\]
By (ii), this limit is the infimum of $\frac{\partial\beta}{\partial t}(x,t)$ over $x\in(0,1)$. Thus $\frac{\partial\beta}{\partial t}(x,t)>0$ and $\beta(x,t)$ is increasing in $t$.

Suppose that $t_{2}>t_{1}$. By (i),
\[\frac{\partial}{\partial x}\left(\beta(x,t_{2})-\beta(x,t_{1})\right)<0\]
for all $x\in(0,1)$. We now have that $\beta(x,t_{2})-\beta(x,t_{1})$ is positive and monotone decreasing in $x$. The result follows.}\end{proof}

\begin{proof}[{Proof of Lemma \ref{R 1}}]
{For any $x\in(0,1)$, combining Lemmas \ref{alpha 1} and \ref{beta all x},
\begin{align*}
F\left(x,\frac{1}{2}\right)-F(x,1)&=\left(\alpha\left(\frac{1}{2}\right)-\alpha(1)\right)+\left(\beta(x,1)-\beta\left(x,\frac{1}{2}\right)\right)\\
&>\left(\alpha\left(\frac{1}{2}\right)-\alpha(1)\right)\;\;>\;\;0.
\end{align*}
Referring to the description of $R_{t}$ in Lemma \ref{coord Q F}, this inequality implies the containment $R_{1}\subset R_{\frac{1}{2}}$ with containment of vertical sides as desired.}
\end{proof}

\begin{proof}[{Proof of Lemma \ref{R 2/5}}]
{For any $x\in(0,1)$, combining Lemmas \ref{alpha beta 2/5} and \ref{beta all x},
\begin{align*}
F\left(x,\frac{1}{2}\right)-F\left(x,\frac{2}{5}\right)&=
\left(\alpha\left(\frac{1}{2}\right)-\alpha\left(\frac{2}{5}\right)\right)-
\left(\beta\left(x,\frac{1}{2}\right)-\beta\left(x,\frac{2}{5}\right)\right)\\
&>
\left(\alpha\left(\frac{1}{2}\right)-\alpha\left(\frac{2}{5}\right)\right)-
\left(\beta\left(0,\frac{1}{2}\right)-\beta\left(0,\frac{2}{5}\right)\right)\\
\\
&>\;\;0.
\end{align*}
The containment $R_{\frac{2}{5}}\subset R_{\frac{1}{2}}$ follows as above.}
\end{proof}

\section{A Family of Finite Covers of $M$}
\label{covers}
\begin{proof}[{Proof of Corollary \ref{cor}}]
{For each genus $g$, there is a finite cover $\Sigma_{g}\to\Sigma_{2}$. This can be seen by arranging $\Sigma_{g}$ with one `hole' in the center, and with the others forming a wheel around the center. Now rotation by $\frac{2\pi}{g-1}$ is an evident topological symmetry of the picture, and the quotient is $\Sigma_{2}$. The family $\left\{M_{n}\right\}_{n=0}^{\infty}$ is generated from this family of covers, as we now describe.

For each $g\ge2$, the annuli in $P$ lift to annuli $P_{g}\subset\Sigma_{g}$. Denote $\Sigma':=\Sigma_{g}\setminus P_{g}$ (and recall $\Sigma=\Sigma_{2}\setminus P$). We choose a fundamental domain for the covering $\Sigma'\to\Sigma$, as pictured in Figure \ref{coverpiece}, and denote it $B$. To do this carefully, one should choose a two-holed torus fundamental domain for $\Sigma_{g}\to\Sigma_{2}$, include only one of the two boundary components, and take the complement of $P_{g}$. Call this piece $B$. Because $B$ is connected and $\Sigma'$ is the union of finitely many copies of $B$, it is clear that the cover $\Sigma'$ is connected.

\begin{figure}[h]
	\begin{minipage}[]{.45\linewidth}
	\centering
	\includegraphics[height=5cm]{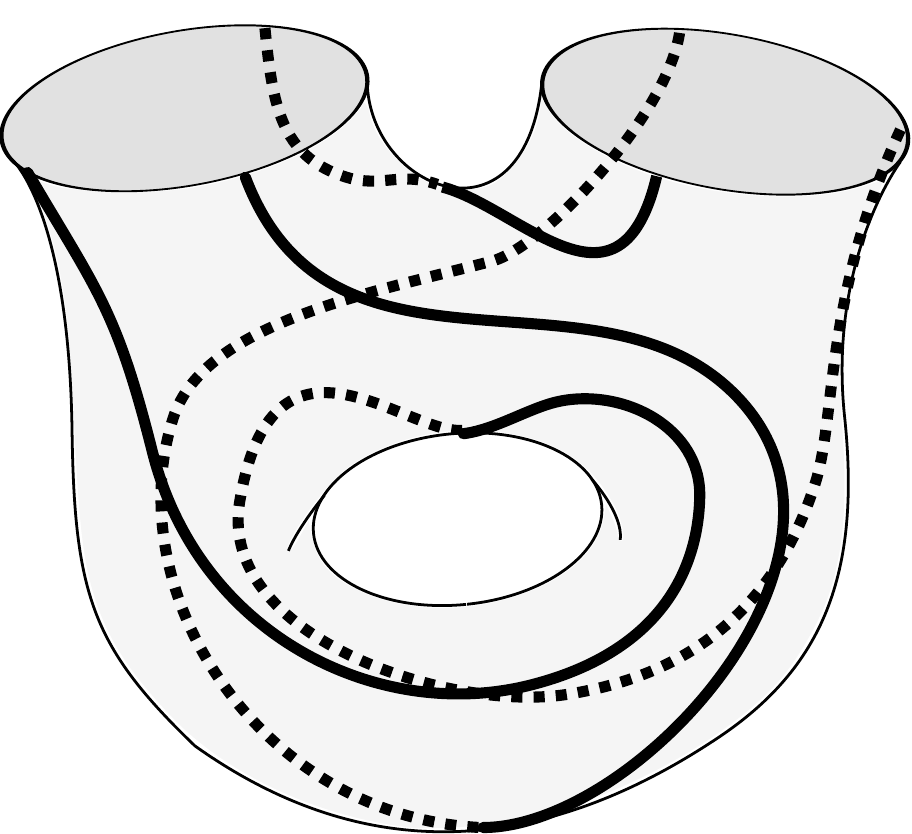}
	\caption{The fundamental domain $B$ for the covering $\Sigma'\to\Sigma$.}
		\label{coverpiece}
	\end{minipage}
	\hspace{.5cm}
	\begin{minipage}[]{.42\linewidth}
	\centering
	\includegraphics[width=6cm]{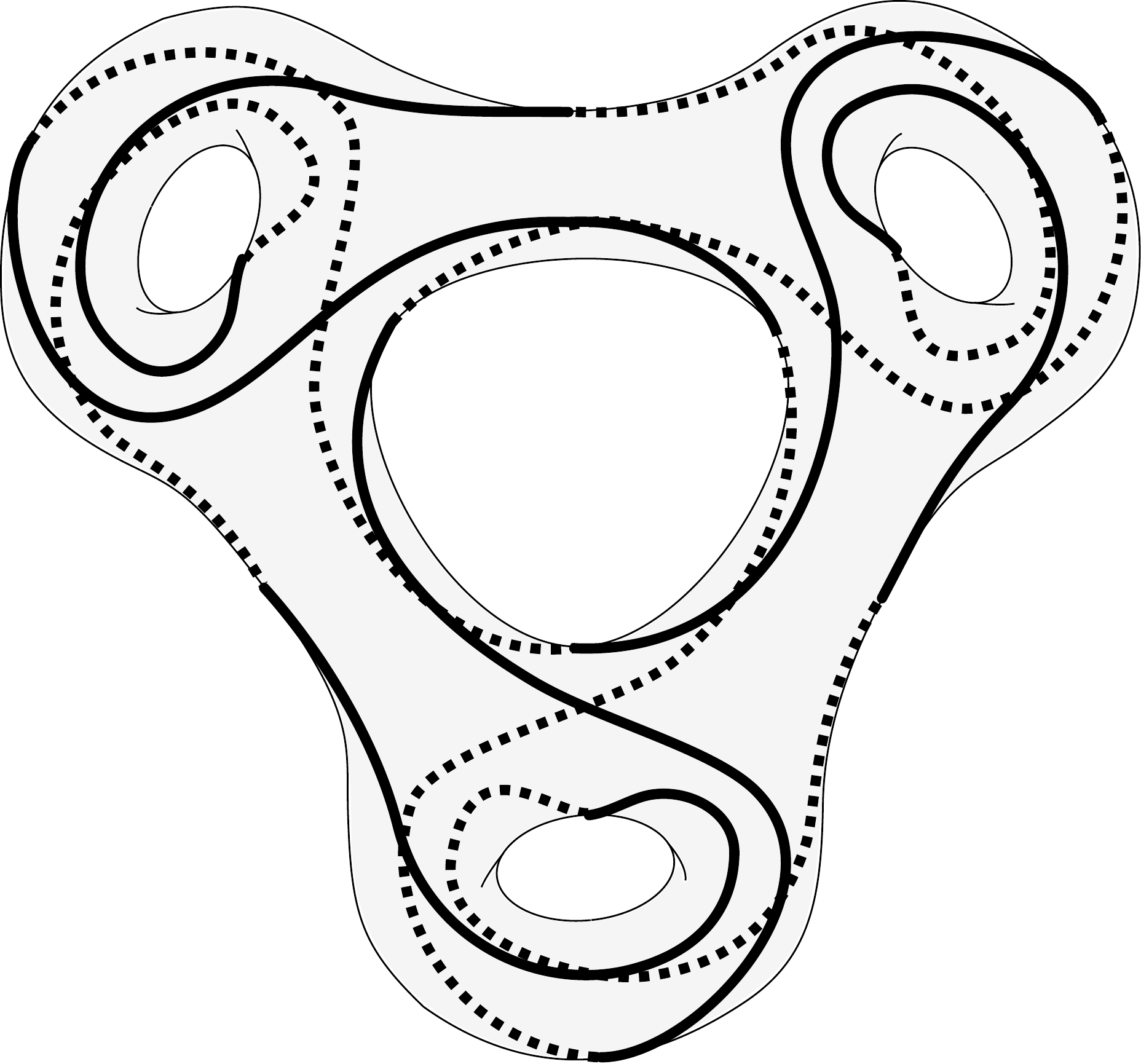}
	\caption{The cover $M_{4}$, the pared 3-manifold $(H_{4},P_{4})$. Note that $P_{4}$ has two connected components.}
		\label{coverpic}
	\end{minipage}
\end{figure}

The intersection of $P_{g}$ with the two-holed torus consists of three connected components. The gluing pattern of $g-1$ copies of $B$, in order to build $\Sigma'$, matches these components in a straightforward fashion. Upon inspection it is clear that $P_{2g}$ has two components and $P_{2g+1}$ has three. (See Figure \ref{coverpic} for a picture of $P_{4}$).

Recalling that $H_{g}$ indicates the closed genus $g$ handlebody, we consider the pared 3-manifolds $M_{n}:=\left(H_{n},P_{n}\right)$. The boundary $\partial M_{2g}$ is a genus $2g$ surface, with two non-separating disk-busting annuli deleted, and is thus homeomorphic to $\Sigma_{2g-2,4}$. Similarly, $\partial M_{2g+1}\cong\Sigma_{2g-2,6}$. Applying Proposition \ref{cover}, the corollary follows.}
\end{proof}

\section{Appendix}

Lemmas \ref{t0} through \ref{A7}, as they can clearly be verified by a finite list of computations involving integers and standard functions, are available in a Mathemetica notebook\footnote{as an ancillary file at arxiv.org, or at math.uic.edu/\textasciitilde gaster/}.

\begin{lemma}
\label{t0}
{$\frac{1}{2}\left(5+3\sqrt{3}-\sqrt{44+26\sqrt{3}}\right)<\frac{2}{5}$}
\end{lemma}

\begin{lemma}
\label{A1}
{$\frac{11}{25}<\cos(\frac{7\pi}{20})$}
\end{lemma}
\begin{lemma}
\label{A2}
{$\frac{1136}{4205}>\cos(\frac{5\pi}{12})$}
\end{lemma}
\begin{lemma}
\label{A3}
{$\frac{5}{9}>\cos(\frac{19\pi}{60})$}
\end{lemma}
\begin{lemma}
\label{A4}
{$\frac{116}{225}<\cos(\frac{13\pi}{40})$}
\end{lemma}
\begin{lemma}
\label{A5}
{$25^{20}(7+3\sqrt{5})^{27}>2^{27}(137+36\sqrt{14})^{20}$}
\end{lemma}
\begin{lemma}
\label{A7}
{$6728^{43}(137+36\sqrt{14})^{46}<25^{46}(43897+225\sqrt{37169})^{43}$}
\end{lemma}

\end{document}